\documentclass[]{interact}

\usepackage{epstopdf}%
\usepackage[caption=false]{subfig}%

\usepackage[numbers,sort&compress]{natbib}%
\bibpunct[, ]{[}{]}{,}{n}{,}{,}%

\theoremstyle{plain}%
\newtheorem{theorem}{Theorem}[section]
\newtheorem{lemma}[theorem]{Lemma}

\newtheorem{proposition}[theorem]{Proposition}

\theoremstyle{definition}
\newtheorem{definition}[theorem]{Definition}
\newtheorem{example}[theorem]{Example}

\theoremstyle{remark}
\newtheorem{remark}{Remark}

\usepackage{amsmath,amssymb,amsfonts}
\usepackage{tikz}
\usepackage{pgfplots}
\usetikzlibrary{arrows.meta}
\pgfplotsset{compat=newest}
\usepackage{tikz-network}
\usepackage{nicefrac}
\usepackage{microtype}

\begin{document}

\articletype{ARTICLE TEMPLATE}%

\title{Nonlocal diffusion of variable order on complex networks}

\author{
\name{Daniele Bertaccini\textsuperscript{a}\thanks{CONTACT Daniele
Bertaccini. Email: bertaccini@mat.uniroma2.it}\textsuperscript{b}
and F. Durastante\textsuperscript{c}}
\affil{\textsuperscript{a}Universit\`{a} di Roma Tor Vergata,
Dipartimento di Matematica\\ \textsuperscript{b}Consiglio Nazionale
delle Ricerche, Istituto per le Applicazioni del Calcolo \lq\lq M.
Picone\rq\rq, Roma, Italy\\ \textsuperscript{c}Universit\`{a} di
Pisa, Dipartimento di Matematica. Largo Bruno Pontecorvo, 5, 56127
Pisa, Italy} }

\maketitle

\begin{abstract}
Some aspects of nonlocal dynamics on directed and undirected
networks for an initial value problem whose Jacobian matrix is a
variable-order fractional power of a Laplacian matrix are discussed
here. This is a new extension to non-stationary behavior of a class
of non-local phenomena on complex networks for which
both directed and undirected graphs are considered. Under
appropriate assumptions, the existence, uniqueness, and uniform
asymptotic stability of the solutions of the underlying initial
value problem are proved. Some examples giving a sample of the
behavior of the dynamics are also included.
\end{abstract}


\section{Introduction}

To model the evolution of probability distributions on graphs, two
popular approaches are represented by the heat equation
\begin{equation}\label{eq:heatequation}
\begin{cases}
\mathbf{p}'(t) = -  \mathbf{p}(t) J, & t > 0, \\
\mathbf{p}(0) = \mathbf{p}_0, & \sum_{j=1}^{n} (\mathbf{p}_0)_j = 1,
\end{cases}
\end{equation}
and by the Schr\"odinger equation
\begin{equation}\label{eq:schrodinger}
\begin{cases}
\boldsymbol{\psi}'(t) = - i\,\boldsymbol{\psi}(t) J , & t > 0, \\
\boldsymbol{\psi}(0) = \boldsymbol{\psi}_0, & \sum_{j=1}^{n} |(\boldsymbol{\psi}_0)_j|^2 = 1,
\end{cases}
\end{equation}
where $\mathbf{p}(t)\geq 0$ is a time-dependent probability
distribution in \eqref{eq:heatequation}, and, in
\eqref{eq:schrodinger}, the probability that a particle at time $t$
is at node $v_j \in V$, $V$ the set of all the nodes, is
$|\psi_j(t)|^2/\sum_{j=1}^{n} |\psi_j(t)|^2$. The matrix $J$ is
either the combinatorial Laplacian matrix $L$ associated to the
graph $G=(V,E)$ and to its adjacency and degree matrices $A$ and
$D$, i.e., the singular M-matrix
\begin{equation}\label{eq:combinatorial_laplacian}
L = D - A, \qquad D = \operatorname{diag}(A\mathbf{1}), \quad (A)_{l,j} = a_{l,j} = \begin{cases}
1, & (v_l,v_j) \in E,\\
0, & \text{otherwise}.
\end{cases}
\end{equation}

The solutions of equations~\eqref{eq:heatequation}
and~\eqref{eq:schrodinger} produce a probability
distribution/density at each time step $t$, with the caveat that for
~\eqref{eq:schrodinger} we need to consider the amplitudes
$|\psi_j(t)|^2$, $j=1,...,n$, and model the evolution of the
probability distribution for a random walker on $G$ that moves
between adjacent nodes, that is locally. Recently, there has been an
interest in expanding these navigational strategies to cover also
the evolution of probability distributions on $G$ for walkers that
perform ``long jumps'', i.e., L\'evy flights, on the underlying
graph. To achieve this result, generalizations
of~\eqref{eq:heatequation} and~\eqref{eq:schrodinger} in which the
combinatorial Laplacian matrix is substituted by either its
fractional power $L^{\alpha}$, $\alpha \in (0,1]$, have been
proposed;
see~\cite{BenziBertacciniDurastanteSimunec,RiascosPhysRevE,PhysRevE.102.022142},
or by the generalized $k$-path Laplacian matrix $L_G$;
see~\cite{MR2900722,MR3624678,MR3834211}. In both cases, the choice
of the type of jumps we want our walker to perform on $G$ are
selected to be the same for all times $t$, while it could be more
natural to think about a walker following non-stationary jumps
instead. Sometimes, it can decide to explore the network locally,
some others its attention span is diminished and opts to start
skipping the nodes by performing longer jumps. To model this
behavior, we consider the nonautonomous extensions
of~\eqref{eq:heatequation} and~\eqref{eq:schrodinger} given~by
\begin{equation}\label{eq:nonautonomous_extension}
\begin{cases}
\mathbf{p}'(t) = -\mathbf{p}(t) \mathcal{L}(t) , & t > 0, \\
\mathbf{p}(0) = \mathbf{p}_0,
\end{cases} \quad
\begin{cases}
\boldsymbol{\psi}'(t) = - i\,\boldsymbol{\psi}(t)\mathcal{L}(t) , & t > 0, \\
\boldsymbol{\psi}(0) = \boldsymbol{\psi}_0,
\end{cases}
\end{equation}
where $\mathbf{p}(t)\geq 0$ is a time-dependent probability
distribution, the amplitudes associated with $\boldsymbol{\psi}(t)$
are a time-dependent probability distribution, and
$\mathcal{L}(t)$ is obtained as a time-dependent
extension of the fractional Laplacian~$L^\alpha$.

In this work we aim to analyze this generalization {based on
a variable fractional exponent, potentially able to model better the nonlocal behavior of
the underlying models} and discuss some
theoretical properties of these models. Moreover, numerical
integration methods for these problems are also considered with an
eye to more efficient ways to compute the associated matrix-function
vector products in view of solving much larger problems.

{In the remaining part of the introduction,
Section~\ref{sec:notation_for_graphs_and_newtorks}, we recall
some notations related to graphs for a network; then in Section~\ref{sec:from_non_local_to_nonautonomous}
we discuss the definitions of
the fractional Laplacian matrix and give some information on the
transformed $k$-path Laplacian}. Then, in
Section~\ref{sec:new_proposal}, we introduce our generalization to
get the non-autonomous extension
in~\eqref{eq:nonautonomous_extension} and discuss some of its
properties. Section~\ref{sec:integrating_the_non_autonomous_system} briefly
introduces the problem of the numerical integration of the
underlying non-autonomous systems and the theoretical analysis is
completed with some numerical tests on real-world complex networks.
{Section~\ref{sec:conclusions} summarizes the obtained result,
and highlights some future research directions. }

\subsection{Notation and graphs for a network}
\label{sec:notation_for_graphs_and_newtorks}

A well known efficient and clear way to represent the complex
interactions of a network is through the use of \emph{graph} models.
A \emph{graph} $G$ is defined by a set of \emph{nodes} (or
\emph{vertices} $V=\{v_1,\ldots,v_n\}$ and a set of \emph{edges} $E$
that are a subset of the Cartesian product $E \subseteq V \times V$.
We set $G = (V,E)$. Cartesian products are ordered, thus if $G$ is
an \emph{undirected graph}, we assume that whenever $(v_l,v_j) \in
E$ then $(v_j,v_l) \in E$, otherwise $G$ it is a \emph{directed
    graph}. To avoid repeating this specification, we denote as
$\{v_l,v_j\}$ the unordered pairs.

A weighted (undirected) graph $G =(V,E,W)$ is then obtained by
considering a (symmetric) weight matrix $W$ with nonnegative entries
$(W)_{i,j}=w_{i,j} \geq 0$ and such that
$w_{i,j} > 0$ if and only if $(v_i,v_j)$ is an edge of $G$. If all
the nonzero weights have value $1$ we omit the weight
specification.

We call a \emph{walk} in $G$ a sequence of edges which joins a sequence
of vertices in $V$. if all vertices (and thus all edges) in the walk
are distinct we call it a \emph{path}. In case of a direct graph,
all the edges in a path should point in the same direction. An
undirect graph is \emph{connected} if for each distinct pairs of
nodes there is a walk between them. A directed graph is
\emph{strongly connected} if for each distinct pairs of nodes $v_i$,
$v_j$, there is a direct walk from $v_i$ to $v_j$. The (geodesic)
\emph{distance} $d(u,v)$ between two vertices $u,v \in V$ is defined
as the length of the shortest path connecting them, where the length
of a path is intended as the number of edges crossed. Observe that,
in the direct case, $d(u,v)$ can be different from $d(v,u)$.
Therefore, in that case, $d$ is only a pseudo distance. We call the
\emph{diameter} of the graph $G$ with respect to the geodesic
distance $d$ the quantity $d_{\text{max}} = \max_{u,v \in V}d(u,v)$,
i.e., the length of the longest shortest path.

For a direct and an undirect graph $G$, we introduce the
\emph{adjacency matrix} $A$ as the $n \times n$ matrix with entries
\begin{equation*}
(A)_{i,j} = a_{i,j} = \left\lbrace\begin{array}{cc}
1, & \text{ if }(v_i,v_j) \in E,  \\
0, & \text{ otherwise}.
\end{array}\right.
\end{equation*}
The adjacency matrix $A$ of an undirected graph $G$ is always
symmetric. In particular, if $G = (V,E)$ is a graph,
given two nodes $u,v \in V$, we say that $u$ is adjacent to $v$ and
write $u \sim v$, if $(u,v) \in E$. The above relation is
symmetric if $G$ is an undirected graph, while in general it is not for a
directed graph. Note that for an unweighted graph we have $W=A$.

We introduce also the \emph{incidence matrix} of an undirected graph
as the $|V|\times |E|$ matrix $B$, defined by $B_{ij}=1$ if the
vertex $v_{{i}}$ and edge $e_j$ are incident and 0 otherwise. For
the incidence matrix of a directed graph an arbitrary sign
convention has to be imposed. We assume here that $B_{ij}=-1$ if the
edge $e_{j}$ leaves vertex $v_{{i}}$, 1 if it enters vertex $v_i$
and 0 otherwise. In the weighted case we substitute to the value
$\pm 1$ the weight of the associated edge.

For every node $v \in V$, we introduce also the \textit{degree}
$\deg(v)$ of $v$ as the number of edges leaving \emph{or}
entering~$v$ taking into account their weights
\begin{equation*}
d_i = \deg(v_i) = \sum_{j\,:\, (v_i,v_j) \in E}w_{i,j}.
\end{equation*}
The degree matrix $D$ is then the diagonal matrix whose entries are given by the degrees of the nodes,~i.e.,
\begin{equation*}
\begin{split}
D = & \operatorname{diag}(\deg(v_1),\ldots,\deg(v_n)) =  \operatorname{diag}(d_1,\ldots,d_n).
\end{split}
\end{equation*}
For directed graphs, it is useful to differentiate the degree of a node $v_i$ respectively to
the incoming and outgoing edges. For this reason we consider \textit{in--degrees} and \textit{out--degrees}
\begin{equation*}
d_i^{(\text{in})} = \deg_{\text{in}}(v_i) = \sum_{j\,:\, (v_j,v_i) \in E}w_{j,i}, \qquad d_i^{(\text{out})} = \deg_{\text{out}}(v_i) = \sum_{j\,:\, (v_i,v_j) \in E}w_{i,j},
\end{equation*}
together with the related diagonal matrices
\begin{align*}
D_{\text{in}} =  & \operatorname{diag}(\deg_{\text{in}}(v_1),\ldots,\deg_{\text{in}}(v_n))  =   \operatorname{diag}(d^{(\text{in})}_1,\ldots,d^{(\text{in})}_n),\\ D_{\text{out}} = &  \operatorname{diag}(\deg_{\text{out}}(v_1),\ldots,\deg_{\text{out}}(v_n)) =  \operatorname{diag}(d^{(\text{out})}_1,\ldots,d^{(\text{out})}_n).
\end{align*}

With this notation, we recall some
definitions pertaining to the Laplacian matrix as given in \cite{BenziBertacciniDurastanteSimunec}.

Let $G = (V,E,W)$ be a weighted undirected graph with weight matrix $W$,
weighted degree matrix $D$ and weighted incidence matrix $B$. Then
the \emph{graph Laplacian} $L$ of $G$ is
given by
\[L = D - W = BB^T.\]
The \textit{normalized random walk} version of the graph Laplacian is
\[D^{-1}L = I - D^{-1}W = D^{-1}BB^T,\]
where $I$ is the identity matrix. Observe that $D^{-1}W$ is a row--stochastic matrix,
i.e. it is nonnegative with row sums equal to 1. The \textit{normalized symmetric} version~is
\[D^{-\frac{1}{2}} L D^{-\frac{1}{2}} = I - D^{-\frac{1}{2}} W D^{-\frac{1}{2}}.\]
If $G$ is unweighted then $W=A$ in the above definitions. As in
\cite{BenziBertacciniDurastanteSimunec}, we assume that every vertex
has nonzero degree or that every vertex is not \textit{isolated}.

Let $G = (V,E,W)$ be a weighted directed graph, with degree matrices $D_{\text{out}}$ and
$D_{\text{in}}$. The nonnormalized directed graph Laplacians
$L_{\text{out}}$ and $L_{\text{in}}$ of $G$ are
\begin{equation}\label{eq:in-and-out-laplacians}
L_{\text{out}} = D_{\text{out}} - W, \qquad L_{\text{in}} = D_{\text{in}} - W.
\end{equation}

To define the normalized versions, we need to invert either the
$D_{\text{in}}$ or the $D_{\text{out}}$ matrices, but the absence of
isolated vertices is no longer sufficient to ensure this since there
could be a node with only outgoing or ingoing edges. A first way to
overcome this issue could be imposing that every vertex has at least
one outgoing and one incoming edge, which is rather restrictive.
Otherwise, we could restrict our attention to the subsets of nodes
having an out--degree or in--degree different from zero.

\section{Non-local navigation strategies}
\label{sec:from_non_local_to_nonautonomous}

There exist different approaches to induce a nonlocal probability
transition on the graph $G = (V,E)$. Given the Laplacian matrix $L$
in~\eqref{eq:combinatorial_laplacian}, consider its $\alpha$th power
$L^\alpha$, $\alpha \in (0,1]$, for a symmetric $L$, i.e., for an
undirected $G$. This can be expressed by
decomposition~\cite{RiascosPhysRevE} as
\begin{equation}\label{eq:eigenvaluedecompofL}
L^\alpha = X \Lambda^\alpha X^T,\; \Lambda =
\operatorname{diag}(\lambda_1^\alpha,\ldots,\lambda_n^\alpha), \; 0
= \lambda_1 \leq \lambda_2 \leq \ldots \leq \lambda_n,\; X^TX = I_n.
\end{equation}
For a directed graph, the definition of $f(L) = L^\alpha$ requires
the Jordan canonical form
(see~\cite{BenziBertacciniDurastanteSimunec} for details)
\begin{equation*}
Z^{-1} L Z = J = \operatorname{diag}(J_1,\ldots,J_p), \quad J_k = J_k(\lambda_k) = \begin{bmatrix}
\lambda_k & 1 \\
& \lambda_k & \ddots \\
& & \ddots & 1 \\
& & & \lambda_k
\end{bmatrix} \in \mathbb{C}^{m_k \times m_k},
\end{equation*}
where $Z$ is nonsingular and $m_1 + m_2 + \ldots + m_p = n$, and
\begin{equation}\label{eq:jordanpart1}
f(L) = Z f(J) Z^{-1} = Z \operatorname{diag}(f(J_1),\ldots,f(J_p))Z^{-1},
\end{equation}
\begin{equation}\label{eq:jordanpart2}
f(J_k) = \begin{bmatrix}
f(\lambda_k) & f^{(1)}(\lambda_k) & \ldots & \frac{f^{(m_k-1)}(\lambda_k)}{(m_k-1)!} \\
& f(\lambda_k) & \ddots & \vdots \\
& & \ddots & f^{(1)}(\lambda_k) \\
& & & f(\lambda_k)
\end{bmatrix},
\end{equation}
that needs the function $f(z)=z^\alpha$ to be defined on the
spectrum of $L$, i.e., the values
\begin{equation*}
f^{(j)}(\lambda_i), \qquad j=0,1,\ldots,n_i-1, \quad i = 1,\ldots s,
\end{equation*}
should exist, where $f^{(j)}$ denotes the $j$th derivative of $f$,
with $f^{(0)} = f$; we refer back
to~\cite{BenziBertacciniDurastanteSimunec} for the details and
to~\cite{HighamBook} for the theory of matrix functions. In both the
directed and undirected case we have that $L^\alpha \mathbf{1} = 0$.

\begin{remark}{\emph{$k$-path Laplacian}.}\label{rmk:k-path-laplacian}
    Another approach  for inferring nonlocal transitions on a network is
    represented by the \emph{transformed} $k$-path \emph{Laplacian}. The
    $k$-path Laplacian matrix $L_k$
    ~\cite{MR2900722,10.1093/comnet/cnx043} of a connected undirected
    graph $G=(V,E)$ is related to the $k$-hopping walks and $k$-path
    degree.
    \begin{description}
        \item[$k$-hopping walk:] A $k$-hopping walk of length $l$ is any sequence of (not necessarily
        different) nodes $v_1,\ldots,v_l,v_{l+1}$ such that
        $d(v_{j},v_{j+1})=k$ for each $j=1,2,\ldots,l$;
        \item[$k$-path degree:] The $k$-path degree $\delta_k(v_l)$ of a node $v_l$ is the number of
        irreducible shortest-paths of length $k$ having $v_l$ as an
        endpoint.
    \end{description}
    The $k$-path analogous of the combinatorial Laplacian
    in~\eqref{eq:combinatorial_laplacian} is defined as the square
    symmetric matrix $L_k$
    \begin{equation*}
    (L_k)_{l,j} = \begin{cases}
    -1, & d(v_l,v_j) = k, \\
    \delta_k(v_l), & v_l \equiv v_j, \\
    0, & \text{otherwise}.
    \end{cases}
    \end{equation*}
    To produce the generalized $k$-path Laplacian inducing the nonlocal
    transition probability, we consider the series
    \begin{equation*}
    L_G = L_1 + \sum_{k \geq 2} \frac{1}{k^\alpha} L_k = L + \sum_{k\geq 2} \frac{1}{k^\alpha} L_k, \qquad \alpha \geq 0,
    \end{equation*}
    that is indeed a finite sum since $L_k$ coincides with the zero
    matrix for each $k > d_{\text{max}}$. Observe that again $L_k
    \mathbf{1} = 0$ $\forall k \geq 1$, and thus $L_G \mathbf{1} = 0$.
    For the remaining part of the manuscript, we will focus on the
    properties of the underlying fractional extension. We refer
    to~\cite{Estrada_2021} for some comparison of the two approaches,
    and to~\cite{MR4216832} for a discussion on using different type of
    series and distances on the graph $G$.
\end{remark}

When we substitute $L^\alpha$ into the heat
equation~\eqref{eq:heatequation} or into the Schr\"odinger equation,
the solution at time $t$ can be expressed in terms of the matrix
exponential~as
\begin{align*}
\mathbf{p}(t) = & \mathbf{p}_0 \exp(-t L^\alpha) ,
\\
\boldsymbol{\psi}(t) = & \boldsymbol{\psi}_0 \exp(-t\, i\, L^\alpha) ,
\end{align*}
that produce a probability distribution $\mathbf{p}(t)$ and a
probability density $|\boldsymbol{\psi}(t)|^2$, respectively,
because
\begin{equation*}
\mathbf{p}(t)\mathbf{1} = \mathbf{p}_0\exp(- t L^\alpha)\mathbf{1}   = \mathbf{p}_0\left(I  -t L^\alpha + \mathbf{p}_0 \frac{t^2 L^{2\alpha} }{2} -\frac{t^3 L^{3\alpha}}{6} + \ldots \right)\mathbf{1}= \mathbf{p}_0\mathbf{1} = 1,
\end{equation*}
having used $L^\alpha \mathbf{1} = \mathbf{0}$, and analogously for
the corresponding Schr\"odinger system. The second noteworthy
property of this characterization of the solution is that for the
heat equation we can easily compute the steady state for $t
\rightarrow +\infty$ for a given $\alpha$ on a graph $G$.

Let us recall some useful definitions of stability from
\cite{HaleBook}. The vector field $f$ below is assumed smooth enough
to ensure existence, uniqueness, and continuous dependence on the
parameters but, in our setting, this will concern only $\alpha(t)$
because $f$ is a linear function of $L^{\alpha(t)}$ and $L$ is
constant.
\begin{definition}[Stability]\label{def:stability}
    Let us consider the differential equation
    \begin{eqnarray}
    \mathbf{x}'(t) = f(t,\mathbf{x}), & f:\mathbb{C}^{n+1} \rightarrow
    \mathbb{C}^n,\\
    f(t,0)=0, & t\geq 0.
    \label{eq:stability1}
    \end{eqnarray}
    \begin{itemize}
        \item The solution $\mathbf{x}=0$ is called \emph{stable} if for any $\varepsilon >
        0$ and any $t_0\geq 0$, there is $\delta = \delta(\varepsilon,t_0)>0$
        such that $||\mathbf{x}_0||<\delta$ implies $||\mathbf{x}(t,t_0,\mathbf{x}_0)||<\varepsilon$
        for $t\geq0$.
        \item The solution $\mathbf{x}=0$ is called \emph{uniformly stable} if it
        stable and $\delta$ can be chosen independent of $t_0\geq 0$.
        \item The solution $\mathbf{x}=0$ is called \emph{asymptotically uniformly stable}
        if it uniformly stable, there exists a $b>0$ that for every
        $\eta>0$ there exists a $T(\eta)>0$ such that $||\mathbf{x}_0||<b$ implies
        $||\mathbf{x}(t,t_0,\mathbf{x}_0)||<\eta$ if $t\geq t_0+T(\eta)$.
    \end{itemize}
    The stability of any nonzero solution of the underlying differential
    equation is easily derived from the ones above. See, e.g.,
    \cite{HaleBook} for other details.
\end{definition}
When we deal with the heat equation in~\eqref{eq:heatequation}, and,
in general, with any constant coefficients case, stability depends
only on the eigenvalues of the Jacobian matrix, i.e., $L^{\alpha}$
in the case of fractional diffusion. Since $L^{\alpha}$ is an
M-matrix (see \cite{BenziBertacciniDurastanteSimunec}) for any value
of $\alpha \in (0,1)$, every eigenvalue of $-L^{\alpha}$ have a
nonpositive real part and those with zero real part have a Jordan
block of size one, i.e., they are semisimple. Thus, the dynamic is
automatically uniformly stable. Moreover, for a connected graph $G$,
we~find
\begin{equation*}
\mathbf{p}(t) = \mathbf{y}_0 \exp(-t L^{\alpha}) \rightarrow
\mathbf{1}^T/n, \quad \alpha \in (0,1).
\end{equation*}

\begin{example}
    Consider a simple cycle graph $G$ with $n=4$ nodes, i.e, $V =
    {1,2,3,4}$, $E = \{ \{1,2\},\{2,3\},\{3,4\},\{4,1\} \}$, see
    Figure~\ref{fig:example1} on the left. The fractional Laplacian
    matrix for $G$ can be computed in closed form as
    \begin{equation*}
    L^\alpha = \begin{bmatrix}
    2^{\alpha -2} \left(2^{\alpha }+2\right) & -4^{\alpha -1} & 2^{\alpha -2} \left(2^{\alpha }-2\right) & -4^{\alpha -1} \\
    -4^{\alpha -1} & 2^{\alpha -2} \left(2^{\alpha }+2\right) & -4^{\alpha -1} & 2^{\alpha -2} \left(2^{\alpha }-2\right) \\
    2^{\alpha -2} \left(2^{\alpha }-2\right) & -4^{\alpha -1} & 2^{\alpha -2} \left(2^{\alpha }+2\right) & -4^{\alpha -1} \\
    -4^{\alpha -1} & 2^{\alpha -2} \left(2^{\alpha }-2\right) & -4^{\alpha -1} & 2^{\alpha -2} \left(2^{\alpha }+2\right) \\
    \end{bmatrix},
    \end{equation*}
    while the $k$-path generalized Laplacian matrix reads as
    \begin{equation*}
    L_G = \begin{bmatrix}
    2^{-\alpha }+2 & -1 & -2^{-\alpha } & -1 \\
    -1 & 2^{-\alpha }+2 & -1 & -2^{-\alpha } \\
    -2^{-\alpha } & -1 & 2^{-\alpha }+2 & -1 \\
    -1 & -2^{-\alpha } & -1 & 2^{-\alpha }+2 \\
    \end{bmatrix}.
    \end{equation*}
    These above are indeed essentially different. Indeed, $L^\alpha
    \xrightarrow[]{\alpha \rightarrow 1} L$ and $L_G
    \xrightarrow[]{\alpha \rightarrow +\infty} L$, while for $\alpha
    \rightarrow 0$, $L_G$ converges to the Laplacian matrix of the
    complete graph, and $L^\alpha$ to $1/4$ of it; see
    Figure~\ref{fig:example1}.
    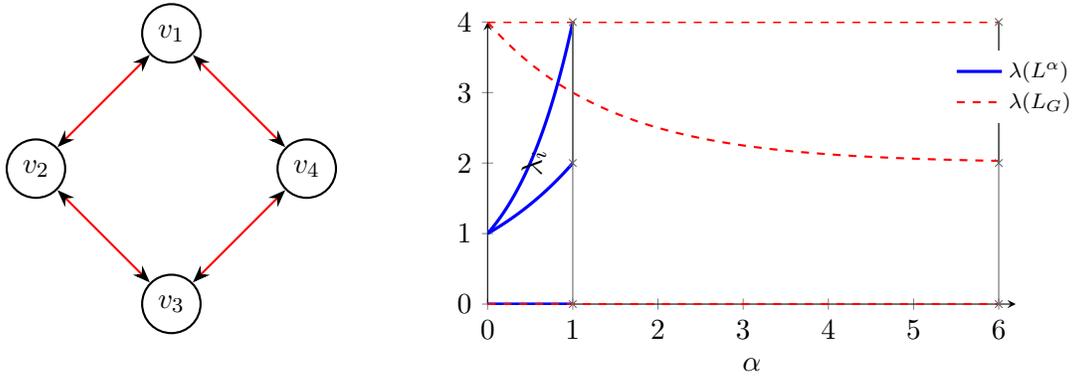
\begin{figure}[htbp]
        \centering
        \begin{tikzpicture}[scale=0.9,baseline=(current bounding box.north)]
        \begin{scope}[every node/.style={circle,thick,draw}]
        \node (A) at (0,2) {$v_1$};
        \node (B) at (-2,0) {$v_2$};
        \node (C) at (0,-2) {$v_3$};
        \node (D) at (2,0) {$v_4$};
        \end{scope}

        \begin{scope}[>={Stealth[black]},
        every node/.style={fill=white,circle},
        every edge/.style={draw=red,thick}]
        \path [<->] (A) edge (B);
        \path [<->] (B) edge (C);
        \path [<->] (C) edge (D);
        \path [<->] (D) edge (A);
        \end{scope}
        \end{tikzpicture}\hfill
        \begin{tikzpicture}[baseline=(current bounding box.north)]
        \begin{axis}[
        width=0.6\columnwidth,
        height=2.1in,
        xmin=0,
        xmax=6.2,
        axis lines = left,
        xlabel = $\alpha$,
        ylabel = {$\lambda_i$},
        y label style={at={(axis description cs:0.13,.5)},anchor=south},
        legend style={at={(1,0.9)}, anchor=north, legend cell align=left, align=left, draw=none,font=\footnotesize}
        ]
        \addplot [
        domain=0:1,
        samples=100,
        color=blue,
        very thick,
        forget plot,
        ]
        {2^x};
        \addplot [
        domain=0:1,
        samples=100,
        color=blue,
        very thick,
        forget plot,
        ]
        {0};
        \addplot [
        domain=0:1,
        samples=100,
        color=blue,
        very thick,
        ]
        {2^(2*x)};
        \addlegendentry{$\lambda(L^\alpha)$}
        \addplot [
        domain=0:6,
        samples=100,
        color=red,
        dashed,
        thick,
        ]
        {4};
        \addlegendentry{$\lambda(L_G)$}
        \addplot [
        domain=0:6,
        samples=100,
        color=red,
        dashed,
        thick,
        ]
        {(2^(1-x))*(2^x+1)};
        \addplot [
        domain=0:6,
        samples=100,
        color=red,
        dashed,
        thick,
        ]
        {0};

        \addplot[mark=x, mark options={solid}, draw=gray!80!black] coordinates{(1,2)(1,4)(1,0)};
        \addplot[mark=x, mark options={solid}, draw=gray!80!black] coordinates{(6,2)(6,4)(6,0)};
        \end{axis}
        \end{tikzpicture}
        \caption{Comparison of the eigenvalues for different values of
            $\alpha$ between the fractional Laplacian $L^\alpha$ and the
            generalized $k$-path Laplacian $L_G$ matrices for the cycle graph on
            the left. The \raisebox{0.25em}{\protect\tikz\protect\draw[very
                thick,blue](0,0)--(0.5,0);} line represents the eigenvalues of the
            $L^\alpha$ for $\alpha \in [0,1]$, while the
            \raisebox{0.24em}{{\protect\tikz\protect\draw[dashed,thick,red](0,0)--(0.5,0);}}
            represents the Eigenvalues for the generalized $k$-path Laplacian
            for $\alpha \in [0,6]$. The \textcolor{gray}{$\times$}
            represents the eigenvalues of the combinatorial Laplacian matrix~$L$.}
        \label{fig:example1}
    \end{figure}
    Moreover, for $t \rightarrow +\infty$ we can look at the asymptotic
    behavior for a given $\alpha$ in
    Figure~\ref{fig:asymptoticstability_example1}, i.e., the stability
    properties of the solution for both the fractional Laplacian
    $L^\alpha$ and the generalized $k$-path Laplacian.
    \begin{figure}[htbp]
        \centering
        \subfloat[With fractional Laplacian matrix $L^\alpha$]{\input{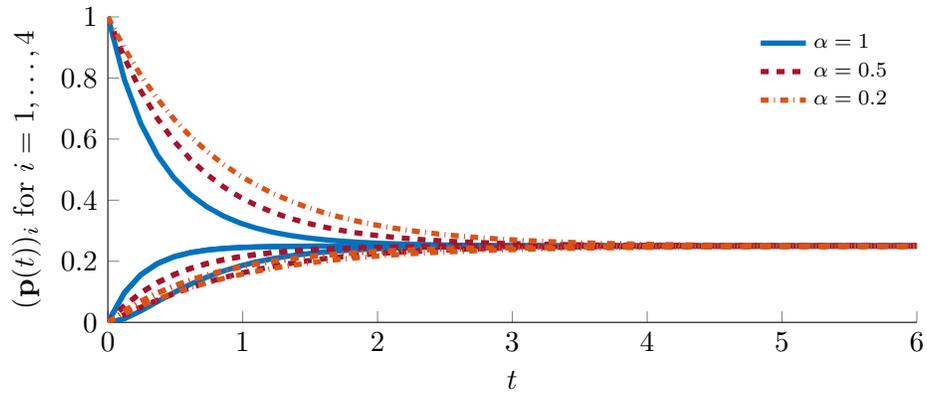}}

        \subfloat[With generalized $k$-path fractional Laplacian matrix]{\input{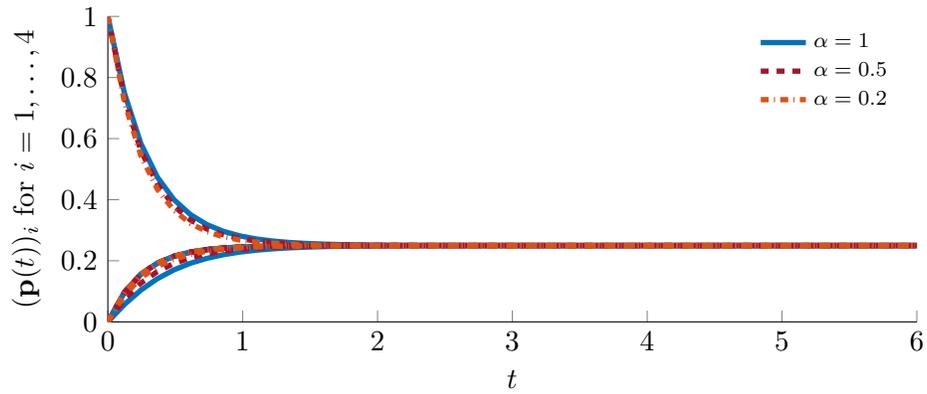}}
        \caption{Comparison of the decay behavior towards the stationary solution
            using the fractional Laplacian $L^\alpha$ and the generalized $k$-path Laplacian.}
        \label{fig:asymptoticstability_example1}
    \end{figure}
    Independently from the value of $\alpha$, we observe that all the
    components of the probability vector reach the stationary solution
    given by the uniformly distributed probability on the nodes.
\end{example}

\begin{remark}
    As discussed in~\cite{2101.00425}, we can interpret the fractional
    navigation strategy as a random walk on a new \emph{complete} graph
    $G'$ that is built on the same nodes of the original graph. Indeed,
    the underlying extension is in general not compatible with the
    dynamics characterizing the original graph model $G$. If one
    restricts the random walks on $G'$ to move on the edges of $G$,
    these are not \emph{stochastically equivalent} to the random walks
    on $G$. In other words, the incompatibility of $G'$ with $G$ means
    that the graph $\hat{G}$ induced by the normalization $D^{-1}L$ can
    not be embedded into the normalized graph $\hat{G}'$ related to
    $D^{-1}_\alpha L^{\alpha}$, for $D_\alpha =
    \operatorname{diag}(L^\alpha)$; refer to the analysis
    in~\cite{2101.00425} for the details.
    For our goals here, this is not an issue.
\end{remark}

\section{Non-local and non-stationary navigation strategies}\label{sec:new_proposal}

We can now modify~\eqref{eq:nonautonomous_extension} by defining the
operator
\begin{equation*}
\mathcal{L}(\alpha(t);t) =  L^{\alpha(t)}, \quad \alpha :  \mathbb{R}_+ \rightarrow (0,1],
\end{equation*}
where we use the function $\alpha(t)$ to modulate the ``quantity of
nonlocality'' at time~$t$ for the probability measures
$\mathbf{p}(t)$ and $|\boldsymbol{\psi}(t)|^2$ over $G$. Such
approach is built by extending the fractional Laplacian matrix to a
variable order, see, e.g.,~\cite{2102.09932,Samko2013,MR2505873}.
Even if this seems a minor change, two main difficulties are
encountered when moving to the non-autonomous setting. In general,
for a system of ordinary differential equations of the form
\begin{equation}\label{eq:generic_ode_problem}
\begin{cases}
\mathbf{y}'(t) = \mathcal{A}(t)\mathbf{y}(t), & t > 0, \\
\mathbf{y}(0) = \mathbf{y}_0,
\end{cases}
\end{equation}
the solution \emph{cannot} be expressed as
\begin{equation*}
\mathbf{y}(t) = \exp(\mathfrak{A}(t)) \mathbf{y}_0,
\end{equation*}
for $\mathfrak{A}(t)$ a primitive of $\mathcal{A}(t)$, consider, e.g., the following example from~\cite{1100529},
\begin{equation*}
\mathcal{A}(t) = \begin{bmatrix}
6 \sin (12 t)-9 \cos ^2(6 t)-1 & \frac{9}{2} \sin (12 t)+6 \cos (12 t)+6 \\
\frac{9}{2} \sin (12 t)+6 \cos (12 t)-6 & -6 \sin (12 t)+\frac{9}{2} \cos (12 t)-\frac{11}{2} \\
\end{bmatrix},
\end{equation*}
for which the solution for $\mathbf{y}_0 = (5,0)^T$ is given by
\begin{equation*}
\mathbf{y}(t) = \begin{bmatrix}
6 \sin (12 t)-9 \cos ^2(6 t)-1 & \frac{9}{2} \sin (12 t)+6 \cos (12 t)+6 \\
\frac{9}{2} \sin (12 t)+6 \cos (12 t)-6 & -6 \sin (12 t)+\frac{9}{2} \cos (12 t)-\frac{11}{2} \\
\end{bmatrix},
\end{equation*}
that is indeed different from $\mathbf{y}(t) = \exp(\mathfrak{A}(t))
\mathbf{y}_0$. Furthermore, we observe also that even if the
eigenvalues of $\mathcal{A}(t)$ are $\lambda_1 = -1$, $\lambda_2 =
-10$ $\forall t$, the solution $\mathbf{x}(t)$ diverges for $t
\rightarrow + \infty$, i.e., all eigenvalues with negative real part
are no longer sufficient to induce the stability of the resulting
system.

The two following sections investigate which properties of the
original systems~\eqref{eq:heatequation} and~\eqref{eq:schrodinger}
can be recovered for the non-autonomous versions using $\alpha$
variable with $t$.

\subsection{{Existence and uniqueness for our non-autonomous extensions}}
Let us analyze first existence and uniqueness of the solution of the
problem whose Jacobian matrix is the fractional power of the
combinatorial Laplacian
\begin{equation}\label{eq:nonautonomousfractionallaplacianproblem}
\begin{cases}
\mathbf{p}'(t) = -\mathbf{p}(t) L^{\alpha(t)}, & t > 0, \\
\mathbf{p}(0) = \mathbf{p}_0, & \sum_{j=1}^{n} (\mathbf{p}_0)_j = 1,
\end{cases} \; \alpha(t) \,:\,\mathbb{R}_+ \rightarrow (0,1].
\end{equation}

\begin{lemma}\label{lem:condition}
    Let $G = (V,E)$ be a simple graph and $\alpha(t) : \mathbb{R}_+
    \rightarrow (0,1]$ a continuous function. Then,
    \begin{enumerate}
        \item $z \mapsto z^{\alpha(t)}$ is defined on the spectrum of the Laplacian matrix $L$ for every
        $t$;
        \item $\exists\, k > 0$ such that $\|L^{\alpha(t)}\|\leq k$ for
        $t>0$;
        \item $f(t,\mathbf{p})=-\mathbf{p} L^{\alpha(t)}$ is Lipschitz continuous with respect to $\mathbf{p}$ for
        $t>0$.
    \end{enumerate}
\end{lemma}

\begin{proof}
    The first item is just a corollary of
    \cite[Theorem~2.6]{BenziBertacciniDurastanteSimunec}, since the
    conditions on the existence of the derivatives are to be considered
    with respect to the $z$ variable, and $\forall\,t > 0$ we have that
    $\alpha(t) \in (0,1]$. For an undirected graph, since the power
    function is monotonically increasing, we bound $\|L^{\alpha(t)}\|_2
    \leq \lambda_{n}^{M} = k$ for $M = \displaystyle \max_{t \in
        \mathbb{R}} \alpha(t)$. For the directed case, and, e.g. the
    out-degree Laplacian, a more involved form for the constant $k$ can
    be obtained \cite[Section~3.2]{BenziBertacciniDurastanteSimunec},
    and \cite[Section~4]{benzi2020rational}. In both cases the obtained
    $k$ can be used as the Lipschitz constant as below
    \[ \| \mathbf{x}L^{\alpha(t)} - \mathbf{y}L^{\alpha(t)} \| \leq k \|\mathbf{x} - \mathbf{y}\|. \]
\end{proof}

\begin{theorem}\label{teo:fractional_Laplacian_existenceuniqueness}
    Under the hypothesis of Lemma~\ref{lem:condition}, there exists a
    unique solution for the Cauchy
    problem~\eqref{eq:nonautonomousfractionallaplacianproblem} for
    $t>0$.
\end{theorem}
\begin{proof}
The system in \eqref{eq:nonautonomousfractionallaplacianproblem} is linear homogeneous with $(L^{\alpha(t)})_{i,j} = l_{i,j}(t) \in \mathcal{C}^0$ a matrix of continuous real functions in $t$. This is indeed a classical case under which existence and uniqueness are guaranteed by, e.g., \cite[Theorem 5.1]{coddington-levinson} or any other book covering fundamentals on ordinary differential equations.
\end{proof}
To build explicitly the solution of
\eqref{eq:nonautonomousfractionallaplacianproblem}, a commutativity
result is of fundamental importance. Indeed, if the matrix
$L^{\alpha(t)}$ commutes with its antiderivative then we can express
the solution by using the exponential of the antiderivative of
$L^{\alpha(t)}$.

\begin{proposition}\label{pro:fractional_Laplacian_commutativity}
    Given a graph $G=(V,E)$, let $L$ be its combinatorial Laplacian
    matrix in~\eqref{eq:combinatorial_laplacian} and $\alpha :
    \mathbb{R}^+ \rightarrow (0,1]$ a continuous function. Then, the
    matrices $-L^{\alpha(t)}$ and $\exp(-\mathfrak{L}(t)) = \exp(-
    \int_0^{t} L^{\alpha(\tau)}\,d\,\tau)$ commute.
\end{proposition}

\begin{proof}
    For an undirected graph $G$, $L$ is a symmetric matrix, then we can
    compute $-L^{\alpha(t)}$ as in~\eqref{eq:eigenvaluedecompofL}. Thus,
    \begin{align*}
    -L^{\alpha(t)} \exp(-\mathfrak{L}(t)) = &\, -X \Lambda^{\alpha(t)} X^T  \exp\left( -\int_0^{t} X\Lambda^{\alpha(\tau)} X^T \,d\,\tau \right)\\  = &\, X \left( -\Lambda^{\alpha(t)} \exp\left(-\int_0^{t} \Lambda^{\alpha(\tau)}\,d\,\tau\right) \right) X^T \\
    = &\, X \left( -\exp\left(\int_0^{t} -\Lambda^{\alpha(\tau)}\,d\,\tau\right) \Lambda^{\alpha(t)}  \right) X^T \\ = &\, -X  \exp\left(-\int_0^{t} \Lambda^{\alpha(\tau)}\,d\,\tau\right) X^T X \Lambda^{\alpha(t)} X^T \\ = &  -\exp(-\mathfrak{L}(t)) L^{\alpha(t)}.
    \end{align*}
    For a directed graph, we need to consider the definition trough the
    Jordan canonical form
    in~\eqref{eq:jordanpart1}-\eqref{eq:jordanpart2}, i.e., we can
    rewrite $-L^{\alpha(t)} \exp(\mathfrak{L}(t))$ as
    \begin{align*}
    -L^{\alpha(t)} \exp(\mathfrak{L}(t)) = & -Z J^{\alpha(t)} Z^{-1} \exp\left(-\int_0^{t} Z  J^{\alpha(\tau)}  Z^{-1}\,d\,\tau \right)\\
    = & Z \left(-J^{\alpha(t)} \exp\left( -\int_{0}^{t} J^{\alpha(t)}\,d\,\tau \right)\right) Z^{-1},
    \end{align*}
    by using Lemma~\ref{lem:condition}, and observing that $f(t) =
    \lambda_k^{\alpha(t)}$ is integrable. Now $\exp(J^{\alpha(t)})$ is a
    matrix exponential of a block-diagonal matrix, thus it is block
    diagonal itself, and the product of the block-diagonal matrix
    commutes if and only if the product of the blocks commutes.
    Therefore, we can reduce the previous computation on a generic
    Jordan block $J_k$. By construction, each block is an upper
    triangular Toeplitz matrix, i.e., is a matrix with constant entries
    along the diagonals. The conclusion follows from the fact that the
    product of upper triangular Toeplitz matrices is commutative because
    it can be expressed as the product of two matrix polynomials for the
    same matrix; see~\cite{bini1994polynomial} and \cite[Chapter
    2]{bertaccini2013complessita}.
\end{proof}

\begin{theorem}\label{thm:solution_of_nonautonomous_heat}
    Given a graph $G=(V,E)$, let $L$ be its combinatorial Laplacian
    matrix in~\eqref{eq:combinatorial_laplacian} and $\alpha :
    \mathbb{R}^+ \rightarrow (0,1]$ a continuous function. Then, the
    problem in~\eqref{eq:nonautonomousfractionallaplacianproblem} admits
    a solution $\mathbf{p}(t)$ on every interval $[0,T]$ and
    $\|\mathbf{p}(t)\|_1  = 1$ $\forall t \in [0,T]$.
\end{theorem}

\begin{proof}
    The proof follows from Theorem
    \ref{teo:fractional_Laplacian_existenceuniqueness} and from the
    property that the matrix $L^{\alpha(t)}$ commutes with its
    antiderivative by Proposition
    \ref{pro:fractional_Laplacian_commutativity}. Indeed, we can express
    the solution in closed form with the exponential of the
    antiderivative of $L^{\alpha(t)}$, i.e., the solution
    of~\eqref{eq:nonautonomousfractionallaplacianproblem} can be
    expressed as
    \begin{equation*}
    \mathbf{p}(t) = \mathbf{p}_0\exp\left(-\int_{0}^{t}L^{\alpha(\tau)}\,d\,\tau\right),
    \end{equation*}
    whenever
    \begin{equation*}
    L^{\alpha(t)} \int_{0}^{t} L^{\alpha(\tau)}\,d\,\tau - \int_{0}^{t}
    L^{\alpha(\tau)}\,d\,\tau  L^{\alpha(t)} = 0,
    \end{equation*}
    that is exactly what we proved in
    Proposition~\ref{pro:fractional_Laplacian_commutativity}. Then, to
    prove that
    \begin{equation*}
    \|\mathbf{p}(t)\|_1 = \mathbf{p}(t)\mathbf{1} = 1,
    \end{equation*}
    we have only to prove that $\mathfrak{L}(t)\mathbf{1} = \mathbf{0}$
    $\forall t \geq 0$, and indeed
    \begin{equation*}
    \mathfrak{L}(t)\mathbf{1} = \int_0^{t} L^{\alpha(\tau)}\,d\,\tau
    \mathbf{1} = \int_0^{t} L^{\alpha(\tau)} \mathbf{1}\,d\,\tau =
    \int_{0}^{t} \mathbf{0}\,d\,\tau = \mathbf{0}.
    \end{equation*}
\end{proof}

\begin{remark}\label{rmk:matrix_function_are_simple}
    The proof of
    Proposition~\ref{pro:fractional_Laplacian_commutativity}, and,
    therefore, of Theorem~\ref{thm:solution_of_nonautonomous_heat},
    depends on one hand on the fact that all the matrices involved are
    diagonalized by the same transform that is independent of $t$, and,
    on the other, the Jordan blocks are upper triangular Toeplitz
    matrices. These properties are inherited from the definition
    in~\eqref{eq:jordanpart1}-\eqref{eq:jordanpart2}. Therefore, both
    the results can be extended to other matrix functions with constant
    L depending on a variable parameter.
\end{remark}

\subsection{{Stability for our non-autonomous extensions}}

Let us discuss the stability
of~\eqref{eq:nonautonomousfractionallaplacianproblem}. The classical
definitions of stability and asymptotic stability, due to Lyapunov,
can be useful for the study of autonomous differential equations.
For our nonautonomous equations, however, the concepts of uniform
stability and uniform asymptotic stability are more appropriate; see
W. A. Coppel \cite[Chapter 1]{coppel1978dichotomies}.

A special case is when $\mathcal{L}(t)$ is a $T$-periodic function
for which it is possible to apply Floquet Theorem, see, e.g.,
\cite[Chapter III.7]{HaleBook}.

\begin{theorem}[Floquet]
    Every fundamental matrix solution $P(T)$ of
    \begin{equation}\label{eq:floquet_matrix}
    P'(t) = - P(t) \mathcal{L}(t), \, t > 0, \; \exists\,T > 0\, :\, \mathcal{L}(t + T) = \mathcal{L}(t), \forall\,t>0,
    \end{equation}
    has the form
    \begin{equation}\label{eq:floquet_solution}
    P(t) = X(t) e^{B t},
    \end{equation}
    where $X(t)$, $B$ are $n \times n$ matrices, $X(t+T)=X(t)$ for all $t$, and $B$ is a constant.
\end{theorem}

The stability studied by the \emph{characteristic exponent}
$\lambda$ in~\eqref{eq:floquet_matrix}, i.e., the complex number
$\lambda$ for which $x(t)e^{\lambda t}$ is a nontrivial solution
of~\eqref{eq:floquet_matrix} with $\mathbf{x}(t) = \mathbf{x}(t +
T)$, implies that there exists a representation of the
solution~\eqref{eq:floquet_solution} for which the values of
$\lambda$ are the eigenvalues of $B$ in~\eqref{eq:floquet_solution}.

\begin{theorem}[{\cite[Theorem 7.2]{HaleBook}}]\leavevmode\label{thm:stability_condition}
    \begin{enumerate}
        \item A necessary and sufficient condition for the system~\eqref{eq:floquet_matrix} to be uniformly stable is
        that the characteristic exponents have real parts $\leq 0$ and the ones with zero real parts have simple elementary divisors.
        \item A necessary and sufficient condition for the system~\eqref{eq:floquet_matrix} to be uniformly
        asymptotically stable is that all the characteristic exponents have
        real parts $< 0$. If this is the case and $P(t)$ is a matrix
        solution of~\eqref{eq:floquet_matrix}, then there exist $K > 0$,
        $\eta
        > 0$ such that $\|P(t)P^{-1}(s)\| \leq K \exp(-\eta(t-s))$, for $t
        \geq s$.
    \end{enumerate}
\end{theorem}

Let us stress that the characteristic exponents are defined only
after the solutions of~\eqref{eq:floquet_matrix} are computed. In
general, there is no straightforward relation between the
characteristic exponents and the eigenvalues of $\mathcal{L}(t)$.

Fortunately, a much more general result is available. Indeed, we can
prove the uniformly asymptotical stability for our linear
differential equation also if $\mathcal{L}(t)$ is not $T$-periodic;
see W. A. Coppel \cite[Chapter 1]{coppel1978dichotomies}, provided
that the exponent scalar function $\alpha(t)$ is regular enough.
\begin{theorem}[\cite{coppel1978dichotomies}]
    If there exist $K > 0$, $\eta> 0$ such that $\|P(t)P^{-1}(s)\| \leq
    K \exp(-\eta(t-s))$, $P$ fundamental matrix solution of
    \eqref{eq:nonautonomousfractionallaplacianproblem} for $t \geq s$,
    the IVP~\eqref{eq:nonautonomousfractionallaplacianproblem} is
    uniformly asymptotically stable. \label{def:UAstability}
\end{theorem}

\begin{theorem}
    Given a (strongly) connected graph $G=(V,E)$, let $L$ be its
    combinatorial (out-degree) Laplacian matrix
    in~\eqref{eq:combinatorial_laplacian} and $\alpha : \mathbb{R}^+
    \rightarrow (0,1]$ a continuous function. Then, the solution
    $\mathbf{p}(t)$ of
    problem~\eqref{eq:nonautonomousfractionallaplacianproblem} is
    uniformly asymptotically stable.
\end{theorem}

\begin{proof}
    Let $P(t)$ be the matrix solution
    of~\eqref{eq:nonautonomousfractionallaplacianproblem}, i.e., $P(t) =
    \exp\left(-\int_{0}^{t}L^{\alpha(\tau)}\,d\,\tau\right)$, we will
    prove that there exists $K > 0$, and $\eta > 0$ such that
    \begin{equation}\label{eq:stability_condition}
    \| P(t)P^{-1}(s) \| \leq K \exp(-\eta(t - s)), \qquad t \geq s,
    \end{equation}
    so that we can apply Theorem~\ref{thm:stability_condition}. We start
    from the case in which $G$ is an undirected graph, thus, by using
    the Euclidean norm, by direct computation, we find
    \begin{align*}
    \| P(t)P^{-1}(s) \|_2 = & \left\| \exp\left(-\int_{0}^{t}L^{\alpha(\tau)}\,d\,\tau\right) \exp\left(\int_{0}^{s}L^{\alpha(\tau)}\,d\,\tau\right)  \right\| \\
    = & \left\| X \exp\left(-\int_{0}^{t}\Lambda^{\alpha(\tau)}\,d\,\tau\right) \exp\left(\int_{0}^{s}\Lambda^{\alpha(\tau)}\,d\,\tau\right) X^{T} \right\| \\
    = & \left\|\exp\left( \int_{t}^{s}\Lambda^{\alpha(\tau)}\,d\,\tau\right) \right\| \\
    \leq & \exp\left(\left\|\int_{t}^{s}\Lambda^{\alpha(\tau)}\,d\,\tau\right\|\right) \leq
    \exp\left( (s-t) \max_{\tau \in \mathbb{R}} \| \Lambda^{\alpha(\tau)} \| \right) \\
    \leq & \exp(- \beta (t-s)),
    \end{align*}
    from which we find~\eqref{eq:stability_condition} for $K=1$ and
    \begin{equation*}
    \eta = \beta = \max_{i,\tau}\lambda_i(L)^{\alpha(\tau)},
    \end{equation*}
    where
    $\lambda_i(L)$ is the $i$-th eigenvalue of the Laplacian. For a
    directed graph, we suppose that the Laplacian is diagonalizable and
    we get the same result but with
    \begin{equation*}
    K= \| Z \|_2 \| Z^{-1} \|_2 = \operatorname{cond}_2(Z),
    \end{equation*}
    $\operatorname{cond}_2$ the 2-norm condition number of the matrix $Z$ diagonalizing $L$.

    If $L$ cannot be diagonalized, a similar approach can be adapted by
    using a Jordan decomposition of the Laplacian matrix.
\end{proof}

    The results on stability analysis we have obtained here could be applied to
        neural networks, by leveraging , e.g., the results in~\cite{agarwal}, and the references
        therein. Other possible applications of these techniques are in~\cite{Agarwal2021,Rajchakit2021,Boonsatit2021,Rajchakit2021b}.

\section{Integrating the non-autonomous systems}\label{sec:integrating_the_non_autonomous_system}

Let us now consider the numerical integration
of~\eqref{eq:nonautonomousfractionallaplacianproblem} to see the
evolution of the probability distributions on the graphs.

Here we use the \textsc{MATLAB}'s ode packages \texttt{ode45} and
\texttt{ode15s}. Both are methods with local error estimators that
can control and adapt the stepsize to reach the prescribed accuracy
(here we request a relative tolerance \texttt{reltol}$=10^{-6}$).
The first method is a fourth-order nonstiff integrator based on
(explicit) Runge-Kutta-Fehlberg formulas while \texttt{ode15s} is
based on the numerical differentiation formulas (NDFs, a
generalization of BDFs; see \cite{lambert1991numerical} for details
on BFDs) of orders from $1$ to $5$. The latter is based on implicit
schemes and is able to manage \emph{stiff} problems. For more
details on formulas for numerical integration of ODEs, in particular
of their consistence, convergence and stability, see, e.g.,
\cite{lambert1991numerical}. In particular, for the definitions of
stiff problem and stiffness see~\cite[Chapter
6]{lambert1991numerical}.

\subsection{Operations with $L^\alpha$}

We considered the use of implicit time-step integration to avoid
possible severe stepsize restrictions to satisfy stability
requirements of the underlying formulas that we experienced in some
tests even for not so long final times. Implicit time-step
integrators for a problem of the form~\eqref{eq:generic_ode_problem}
require the solution of a sequence of linear systems of the form
\begin{equation}\label{eq:thingtosolve}
\left( \alpha_m I + {\delta}_{m} \beta_{m} \mathcal{A}(t_m) \right)
\mathbf{y}^{(m)} = {\delta}_{m} \mathbf{v}^{(m-1)}, \qquad
m=0,1,\dots,n_t,
\end{equation}
where ${\delta}_{m}$ is the $m$-th time step, the coefficients
$\{\alpha_m,\beta_m\}_{m=0}^{\ell}$ are selected depending on the
particular formula used, $\mathbf{v}^{(m-1)}$ is a suitable
combination of the vectors containing previous time steps and the
$\{\alpha_m,\beta_m\}$ coefficients. See, e.g.
\cite[Section~4.1]{MR3989621} for the use of linear multistep
formulas in a similar context and \cite{lambert1991numerical} for a
general discussion.

In the case we are treating here, we use the above mentioned
\textsc{MATLAB} time-step integrators package. In principle, they
can be invoked by employing only the dynamics of the associated
differential problems. This means that we need a routine that can
evaluate $\mathbf{v}\mathcal{A}$ for the $\mathbf{v}$ generated by
the underlying integrator to handle
\begin{equation}\label{eq:f1}
f(x) = -x^{\alpha(t)},
\end{equation}
or
\begin{equation}\label{eq:f2}
f(x) = -i x^{\alpha(t)},
\end{equation}
for computing $\mathbf{v}f(L)$. Then, to formulate and solve the
linear system~\eqref{eq:thingtosolve}, the routine assembles the
whole matrix at each new time-step and utilizes a direct solver. The
building phase for the sequence of linear
systems~\eqref{eq:thingtosolve}, if no further information are given
to the integrator, is completed by performing several matrix-vector
products with $L^{\alpha(t)}$. To reduce the computing time, we
provide to the integrator a routine to build directly the matrix
$L^{\alpha(t)}$ for every $t$.

A further reduction in the timings could result by implementing a
code for computing directly the matrix function-vector product with
\[
F(x;t) = \left( \alpha_m + {\delta}_{m} \beta_{m} x^{\alpha(t)} \right)^{-1}, \quad m=0,1,\dots,n_t,
\]
for $\delta_{m} \mathbf{v}^{(m-1)} F(\mathcal{A};t_m) $. This
procedure can be based on the techniques for the computation of
matrix function times a vector product with the
functions~\eqref{eq:f1}, and~\eqref{eq:f2}. For symmetric positive
definite matrices there exist several efficient approaches, either
based on various type of quadrature
formulas~\cite{MR3597163,MR3987166,MR4046619,MR4099859,MR3504596},
or on Krylov methods~\cite{MR3989621,Massei2021,MR3854059}. The case
we want to deal with here needs an extra care because of the
presence of the zero eigenvalue in $L$ for which we refer to the
strategies introduced in~\cite{benzi2020rational}. Having selected
the procedure for computing the $\alpha$th power, then the extension
to the computation of the $F(x,t)$ could be addressed with the
techniques discussed in~\cite[Section~4.1]{MR3989621}.

\subsection{Numerical examples} We consider numerical examples on some real-world
complex networks from~\cite{nr-aaai15}, and choices for the values
of the function $\alpha(t)$ from~\cite{2102.09932}. All the
experiments run on \textsc{MATLAB} 9.6.0.1072779 (R2019a) installed
on a Linux machine with an Intel\textregistered\ Core\texttrademark\
i7-8750H CPU @ 2.20GHz processor, and 16 Gb of RAM.

\begin{example}
    We consider as the first example of this section the \emph{Zachary's
        Karate club network}~\cite{Beerenwinkel8271,nr-aaai15}. This is a
    small (undirected) social network of a university karate club with
    $n=34$ nodes. We simulate both the dynamics
    from~\eqref{eq:nonautonomous_extension} for some choice of
    $\alpha(t)$ from~\cite{2102.09932}. We start the simulation from a
    random vector sampled from a uniform distribution across the nodes
    of the network.
    \begin{figure}[htbp]
        \centering
        \includegraphics[width=0.8\columnwidth]{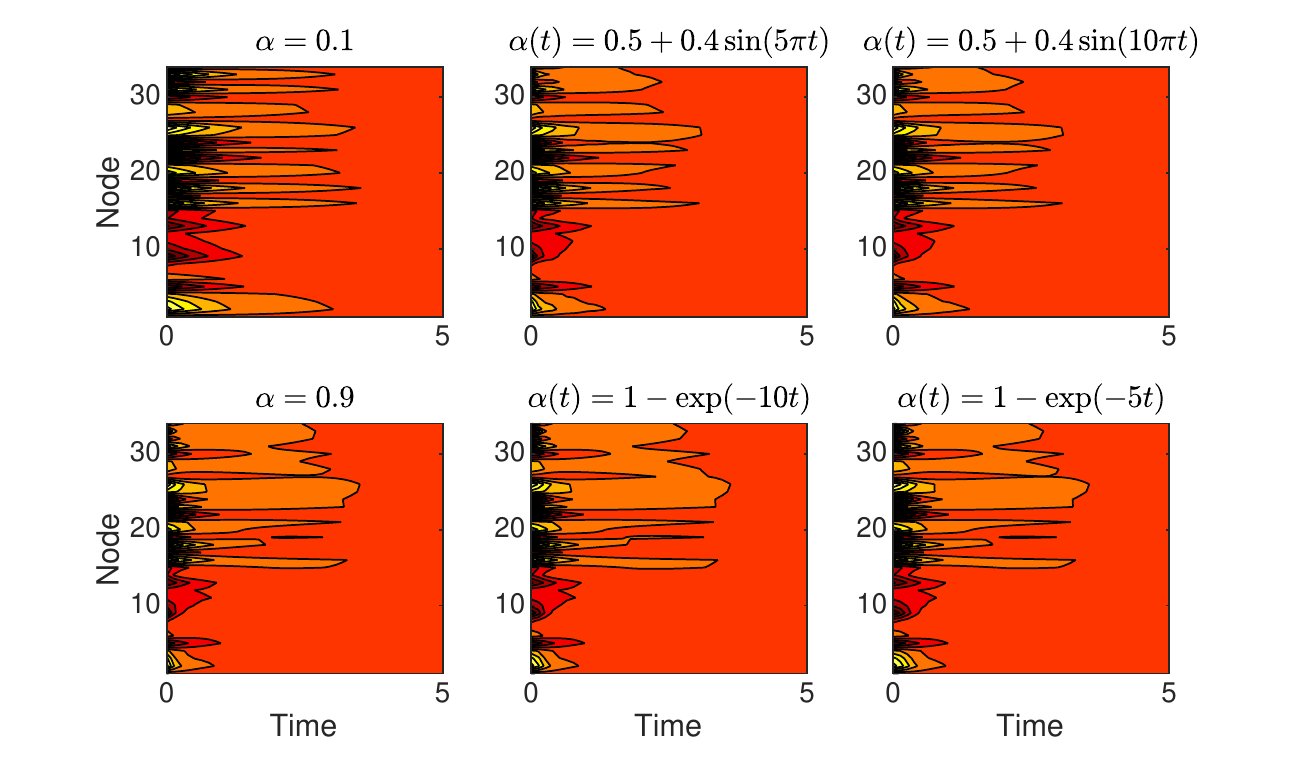}
        \caption{Zachary's karate club network~\cite{nr-aaai15,Beerenwinkel8271}. This is a small social network of a university karate club with $n=34$ nodes. Evolution of the probability $\mathbf{p}_i(t)$ for $i=1,\ldots,34$, $t \in [0,10]$ for different choices of $\alpha(t)$ and the Heat equation dynamics from~\eqref{eq:nonautonomous_extension}.\label{fig:karate1}}
    \end{figure}
    In Figure~\ref{fig:karate1} we observe that the use of the different
    $\alpha(t)$ alters the evolution of the probability. Specifically,
    for each time step, we report the value of the probability at the
    given node. In every case, after a long time, all the solutions
    reach the steady-state, represented by the uniform color, and we
    observe different transient behavior.

    We also observe how this behavior is maintained even in the case of
    \emph{only continuous} $\alpha(t)$ functions. To this end, we
    consider the periodic sawtooth function in Figure~\ref{fig:sawtooth}
    alternating between the values $(0.05,0.75]$.

    \begin{figure}[htbp]
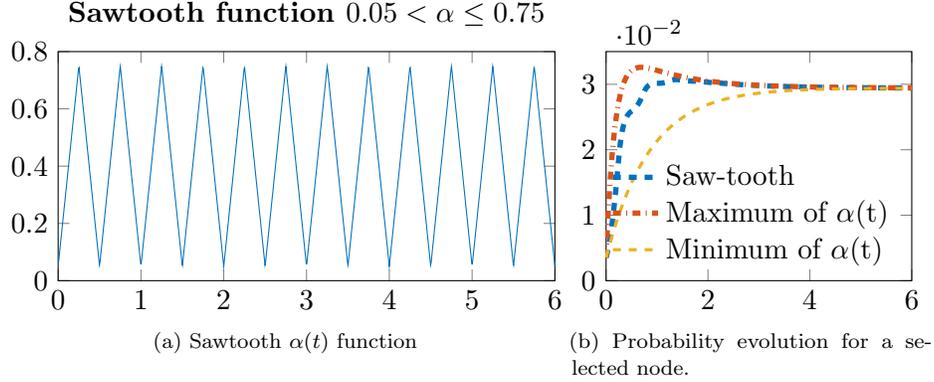

        \centering \subfloat[Sawtooth $\alpha(t)$
        function\label{fig:sawtooth}]{\input{sawtoothfun.tikz}}
        \subfloat[Probability evolution for a selected
        node.\label{fig:sawtooth-dynamic}]{\input{sawtoothdynamic.tikz}}

        \caption{Zachary's Karate club
            network~\cite{nr-aaai15,Beerenwinkel8271}. Example of the
            trajectories obtained by employing a function $\alpha(t)$ that is
            only continuous. Comparison with the trajectories obtained with the
            constant fractional Laplacian with values the maximum and minimum of
            the $\alpha(t)$ function.}
    \end{figure}
    In this case, the evolution of the probability distribution behaves
    again as intended. If we look at the evolution for a given node in
    Figure~\ref{fig:sawtooth-dynamic}, we observe that the probability
    ``oscillates'' between the behavior
    given by the two fixed value of $\alpha(t)$, the two extremes of the
    sawtooth function.

    A more complex behavior is highlighted by Schr\"odinger model, see
    Figure~\ref{fig:karate2}, in which we do not reach a steady-state,
    as in the case of a fixed $\alpha \in (0,1]$.
    \begin{figure}[htbp]
        \centering
        \includegraphics[width=0.8\columnwidth]{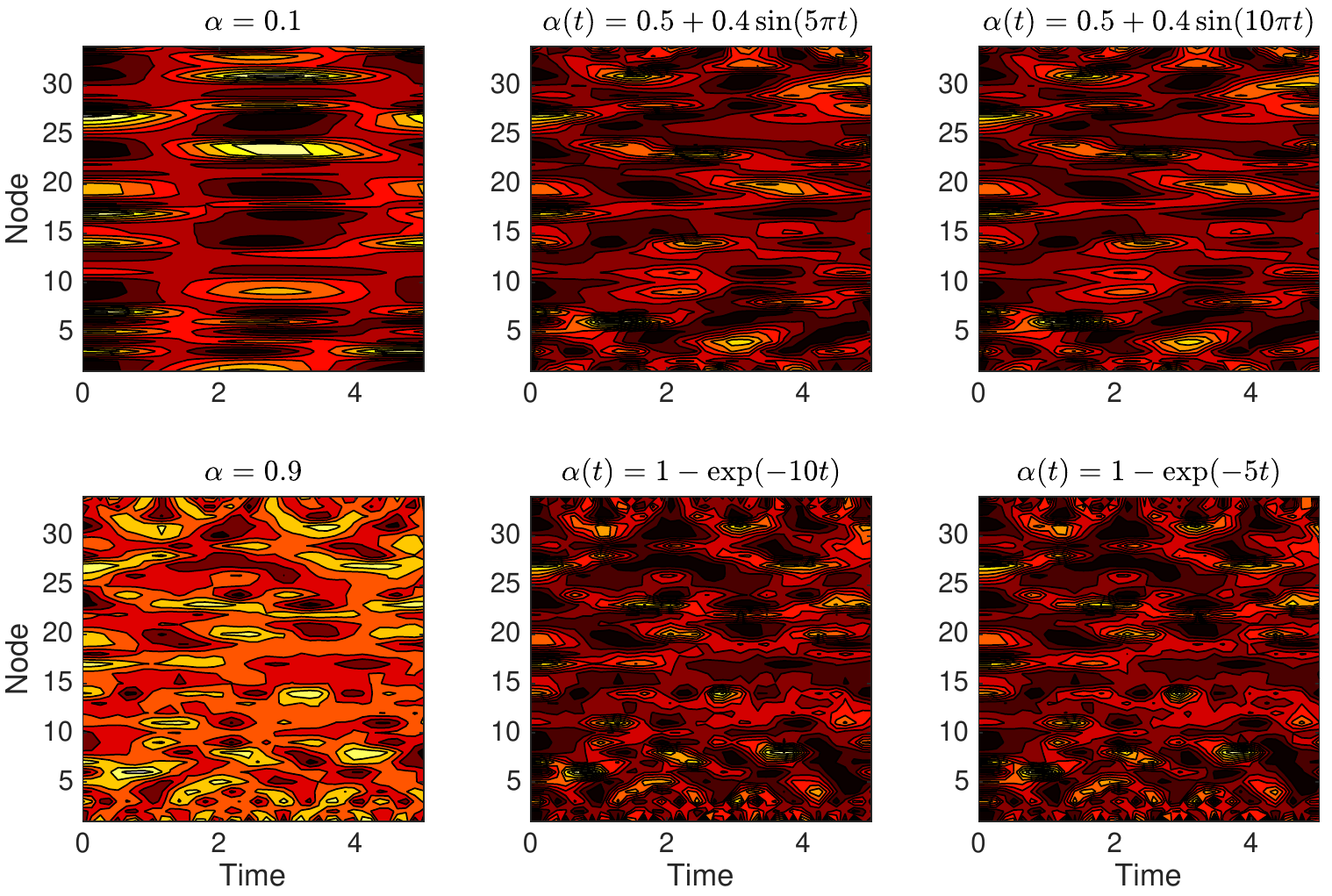}
        \caption{Zachary's Karate club, Schr\"odinger model~\cite{nr-aaai15,Beerenwinkel8271}.
            A small social network of a university Karate club with $n=34$ nodes.\label{fig:karate2}}
    \end{figure}
    For all tests, we have a different behavior in the intermediate
    times, showing the effect of the new exploration strategy.
\end{example}

\begin{example}
    To look more closely at the convergence towards the steady-state, we
    consider a larger graph based on US airlines in 1997, where the
    graph is undirected with $332$ nodes. The (real and nonnegative)
    eigenvalue distribution of the underlying Laplacian matrix is
    depicted in Figure \ref{fig:USAir97eigs}.
    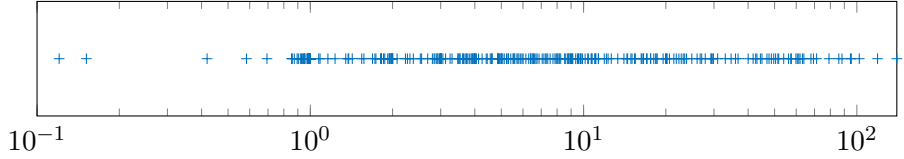
\begin{figure}[htbp]
        \centering
%
%
\definecolor{mycolor1}{rgb}{0.00000,0.44700,0.74100}%
\begin{tikzpicture}

\begin{axis}[%
width=4.5in,
height=0.6in,
at={(1.096in,0.441in)},
scale only axis,
xmode=log,
xmin=0.1,
xmax=140.029428972542,
xminorticks=true,
ymin=-0.20,
ymax=0.20,
ymajorticks=false,
axis background/.style={fill=white},
legend style={legend cell align=left, align=left, draw=white!15!black}
]
\addplot [color=mycolor1, draw=none, mark=+, mark options={solid, mycolor1}]
  table[row sep=crcr]{%
0.0	0\\
0.120484202134577	0\\
0.15156297294081	0\\
0.419139817481331	0\\
0.584111363189358	0\\
0.69412826707162	0\\
0.853692995042642	0\\
0.859167596330982	0\\
0.873802680268994	0\\
0.894836011965865	0\\
0.904242280454085	0\\
0.91929427559303	0\\
0.9282210486742	0\\
0.93493284333123	0\\
0.945321530020794	0\\
0.946777368541647	0\\
0.951057111482639	0\\
0.954084329247929	0\\
0.964500794911579	0\\
0.975660616452569	0\\
0.977171877795232	0\\
0.978265494475648	0\\
0.979033842303663	0\\
0.97997012480046	0\\
0.98074218844082	0\\
0.981879116669073	0\\
0.986773703189752	0\\
0.999999999999991	0\\
0.999999999999992	0\\
0.999999999999994	0\\
0.999999999999994	0\\
0.999999999999994	0\\
0.999999999999995	0\\
0.999999999999995	0\\
0.999999999999995	0\\
0.999999999999996	0\\
0.999999999999996	0\\
0.999999999999996	0\\
0.999999999999997	0\\
0.999999999999997	0\\
0.999999999999997	0\\
0.999999999999997	0\\
0.999999999999998	0\\
0.999999999999999	0\\
0.999999999999999	0\\
1	0\\
1	0\\
1	0\\
1	0\\
1	0\\
1	0\\
1	0\\
1	0\\
1	0\\
1	0\\
1	0\\
1	0\\
1	0\\
1	0\\
1.00000000000001	0\\
1.00000000000001	0\\
1.00000000000001	0\\
1.00000000000001	0\\
1.00000000000001	0\\
1.00000000000001	0\\
1.00000000000002	0\\
1.00000000000002	0\\
1.07360560517169	0\\
1.08552307375568	0\\
1.1597640988245	0\\
1.23180234064629	0\\
1.34779982230458	0\\
1.37502210667577	0\\
1.38196601125011	0\\
1.42414802411377	0\\
1.54257522977715	0\\
1.5708711743253	0\\
1.6893282864596	0\\
1.72342403157658	0\\
1.73819700646828	0\\
1.80867988409741	0\\
1.81272674551409	0\\
1.81956456349598	0\\
1.82188753811355	0\\
1.85335944491435	0\\
1.87310506212169	0\\
1.92002093231431	0\\
1.942989093091	0\\
1.94795600113477	0\\
1.95216887308301	0\\
1.9592655072385	0\\
1.97543347163252	0\\
1.98160441547822	0\\
1.98185023513326	0\\
1.98340119231314	0\\
1.99999999999999	0\\
1.99999999999999	0\\
2	0\\
2	0\\
2	0\\
2.07801650341188	0\\
2.22930541952458	0\\
2.26682335255076	0\\
2.31491554778885	0\\
2.38196601125012	0\\
2.52635208458427	0\\
2.55448824806995	0\\
2.68871555202958	0\\
2.79770297802526	0\\
2.83643868068729	0\\
2.85859338626838	0\\
2.90262994179295	0\\
2.91025933339313	0\\
2.9336841822966	0\\
2.96694647487702	0\\
2.97826526566365	0\\
2.98861034523116	0\\
2.99999999999999	0\\
3	0\\
3	0\\
3	0\\
3	0\\
3	0\\
3.00000000000001	0\\
3.00000000000001	0\\
3.00000000000001	0\\
3.00000000000004	0\\
3.01365430849738	0\\
3.05583238366984	0\\
3.12633475537831	0\\
3.24670496098047	0\\
3.30147718164208	0\\
3.46217364362168	0\\
3.49286389220737	0\\
3.50780284568035	0\\
3.53602259304589	0\\
3.61274669570895	0\\
3.61803398874992	0\\
3.63853496701963	0\\
3.68274404365948	0\\
3.71760466056692	0\\
3.73157622616858	0\\
3.77075548654039	0\\
3.81957134932161	0\\
3.87473001352956	0\\
3.90249524219553	0\\
3.9251449012999	0\\
3.96589813014398	0\\
3.97776298604705	0\\
4	0\\
4	0\\
4.00000000000001	0\\
4.00000000000001	0\\
4.00000000000004	0\\
4.00000000000006	0\\
4.05318592977344	0\\
4.12768809095904	0\\
4.25295642878518	0\\
4.40226921873112	0\\
4.44045898943551	0\\
4.55328712877099	0\\
4.60227344908387	0\\
4.61803398874992	0\\
4.62881418252353	0\\
4.85114349602662	0\\
4.85333802552777	0\\
4.86857147742434	0\\
4.88823338758752	0\\
4.89868715048356	0\\
4.91016005449911	0\\
4.9456145545106	0\\
5	0\\
5.00000000000005	0\\
5.00000000000006	0\\
5.05980692660591	0\\
5.12525914950656	0\\
5.2223508645374	0\\
5.29652610219582	0\\
5.30028260989053	0\\
5.30734762310169	0\\
5.41594494281067	0\\
5.51691112038363	0\\
5.55718703199031	0\\
5.56312532877341	0\\
5.60356883698267	0\\
5.70302316065706	0\\
5.75057055966407	0\\
5.8017516886297	0\\
5.88604420297927	0\\
5.9548052915524	0\\
6.00876764297483	0\\
6.15026344637092	0\\
6.29938895498078	0\\
6.39108073500586	0\\
6.41010661922335	0\\
6.48290327749868	0\\
6.53269159955335	0\\
6.56114054750698	0\\
6.62165988583182	0\\
6.69444530378751	0\\
6.75948301798998	0\\
6.87080663627652	0\\
7.05792365706972	0\\
7.15975219773831	0\\
7.27415496676848	0\\
7.29555835111473	0\\
7.39436148177261	0\\
7.5431362137257	0\\
7.7266977345496	0\\
7.85874088772316	0\\
7.93501910217436	0\\
7.99915626946052	0\\
8.01004630498017	0\\
8.12636700819718	0\\
8.15081866095832	0\\
8.24642643102924	0\\
8.52741567458191	0\\
8.56915635488987	0\\
8.71656134220165	0\\
8.77281722717042	0\\
8.83226464739922	0\\
8.83715294701378	0\\
8.9431719484195	0\\
8.97734302338168	0\\
9.03153275446699	0\\
9.09823237008208	0\\
9.13091356128868	0\\
9.35799296910319	0\\
9.49752812734109	0\\
9.57933348641066	0\\
9.59986180363998	0\\
9.72067416557323	0\\
9.7421026274276	0\\
9.79392928497633	0\\
9.90444674704323	0\\
10.1324721781024	0\\
10.3595802038157	0\\
10.4098037540343	0\\
10.6573503273251	0\\
10.6975777549303	0\\
10.7453480491931	0\\
11	0\\
11.0078092098025	0\\
11.0492703742236	0\\
11.0922295534586	0\\
11.3192609701679	0\\
11.3443418593102	0\\
11.9039354113401	0\\
12.1201962544402	0\\
12.291762123009	0\\
12.6301734657273	0\\
13.3383857072919	0\\
14.0494191741887	0\\
14.4079191933587	0\\
14.4935351498176	0\\
14.899738519482	0\\
15.0427677005834	0\\
15.3616850726459	0\\
15.4661832437355	0\\
16.2461564453916	0\\
16.3387884848299	0\\
16.5137501933655	0\\
16.5742114236715	0\\
17.0083726681041	0\\
17.0789307277543	0\\
17.5851417797701	0\\
17.6001551003885	0\\
18.0646083813994	0\\
18.3590227192585	0\\
18.4355716781773	0\\
18.6000096513751	0\\
19.5483745509803	0\\
19.9599147581659	0\\
20.0488834795901	0\\
20.1078016481564	0\\
20.2209749234045	0\\
20.6191227404061	0\\
21.3980562365089	0\\
21.843497758899	0\\
22.331091363602	0\\
22.7709902818336	0\\
23.1536480498759	0\\
23.4414717234373	0\\
23.8024878647751	0\\
24.8947476915752	0\\
26.0297618744741	0\\
26.7070225830474	0\\
27.0554212105051	0\\
27.8783696743528	0\\
29.2341706222098	0\\
29.3679653723539	0\\
29.7077278845986	0\\
29.851259263103	0\\
30.9385625053187	0\\
33.0033644122489	0\\
34.6614158882201	0\\
35.9689618750604	0\\
36.750817781561	0\\
40.0252842526783	0\\
41.7244469930105	0\\
42.6416425217999	0\\
43.0946379912914	0\\
44.370676614451	0\\
44.9555619986496	0\\
46.8886455240145	0\\
47.7049467079097	0\\
48.6580428110092	0\\
49.6416882004474	0\\
50.4201875431969	0\\
51.3939258934995	0\\
51.7748080536442	0\\
53.9874084810482	0\\
56.2522639469111	0\\
57.3095131269532	0\\
57.6029679833749	0\\
59.8405176212357	0\\
60.5481876244068	0\\
61.2920840902558	0\\
61.9660367533319	0\\
63.2160061447993	0\\
63.9932296917352	0\\
67.9219779947898	0\\
69.3983984740485	0\\
71.237049702096	0\\
79.0892173568012	0\\
85.9107841565965	0\\
88.2288132867113	0\\
94.7606748840765	0\\
95.4122483058717	0\\
102.079786761902	0\\
119.043939236416	0\\
140.029428972542	0\\
};
\end{axis}
\end{tikzpicture}%
        \caption{Eigenvalue distribution of the Laplacian matrix of the graph of US Airlines in 1997.
            Source of the graph data: \cite{nr-aaai15}}
        \label{fig:USAir97eigs}
    \end{figure}
    In the figures \ref{fig:4USAir97ode15s}, \ref{fig:4USAir97odestiff}
    and \ref{fig:4USAir97odesin-stiff} we report the output and, in
    particular, the evolution of the probability distribution for runs
    with short final times. This shows some interesting performances of
    the functions $\alpha(t)$ considered. Moreover, the
    figures~\ref{fig:4USAir97odestiff}
    and~\ref{fig:4USAir97odesin-stiff} are useful to observe a moderate
    stiffness phenomena that can be present even in small graphs, and
    that is clearly highlighted by the difference in the number of time
    steps generated by the underlying stiff and nonstiff time-step
    integrators. In particular, for $\alpha(t)=1-\exp(-10t)$ in
    Figure~\ref{fig:4USAir97odestiff} we have that \texttt{ode45}
    employs $3349$ time steps, while \texttt{ode15s} needs just $90$. As
    expected, the difference is less pronounced in
    Figure~\ref{fig:4USAir97odesin-stiff}, where the periodic transition
    $\alpha(t)=0.5+0.4\sin(4\pi t)$ has been adopted instead:
    \texttt{ode45} employs 549 time steps while \texttt{ode15s} employs
    191.

    \begin{figure}[htbp]
        \centering
        \includegraphics[width=\columnwidth]{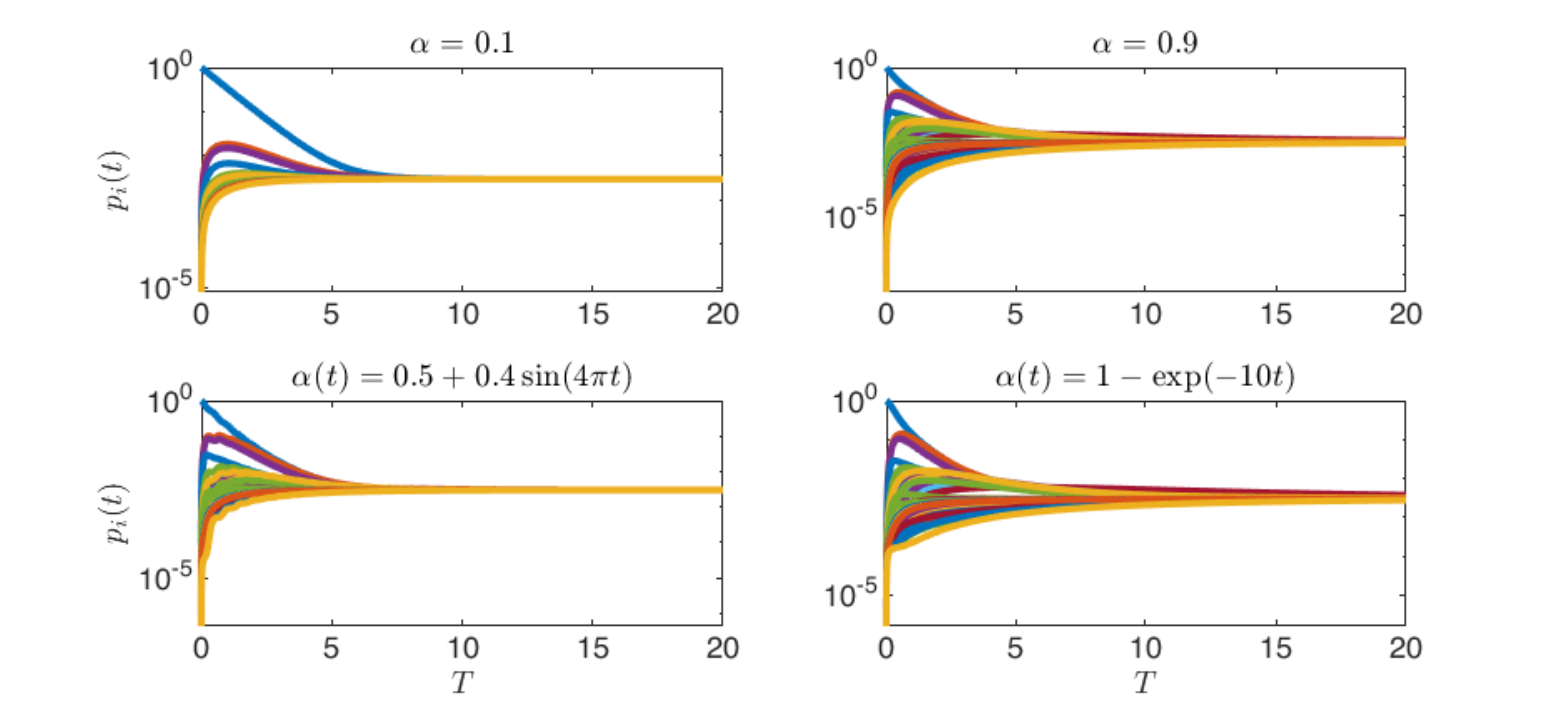}
        \caption{Comparison of short term results of US Airlines in 1997 integrated using ode15s.
            Source of the graph data: \cite{nr-aaai15}.}
        \label{fig:4USAir97ode15s}
    \end{figure}

    \begin{figure}[htbp]
        \centering
        \includegraphics[width=\columnwidth]{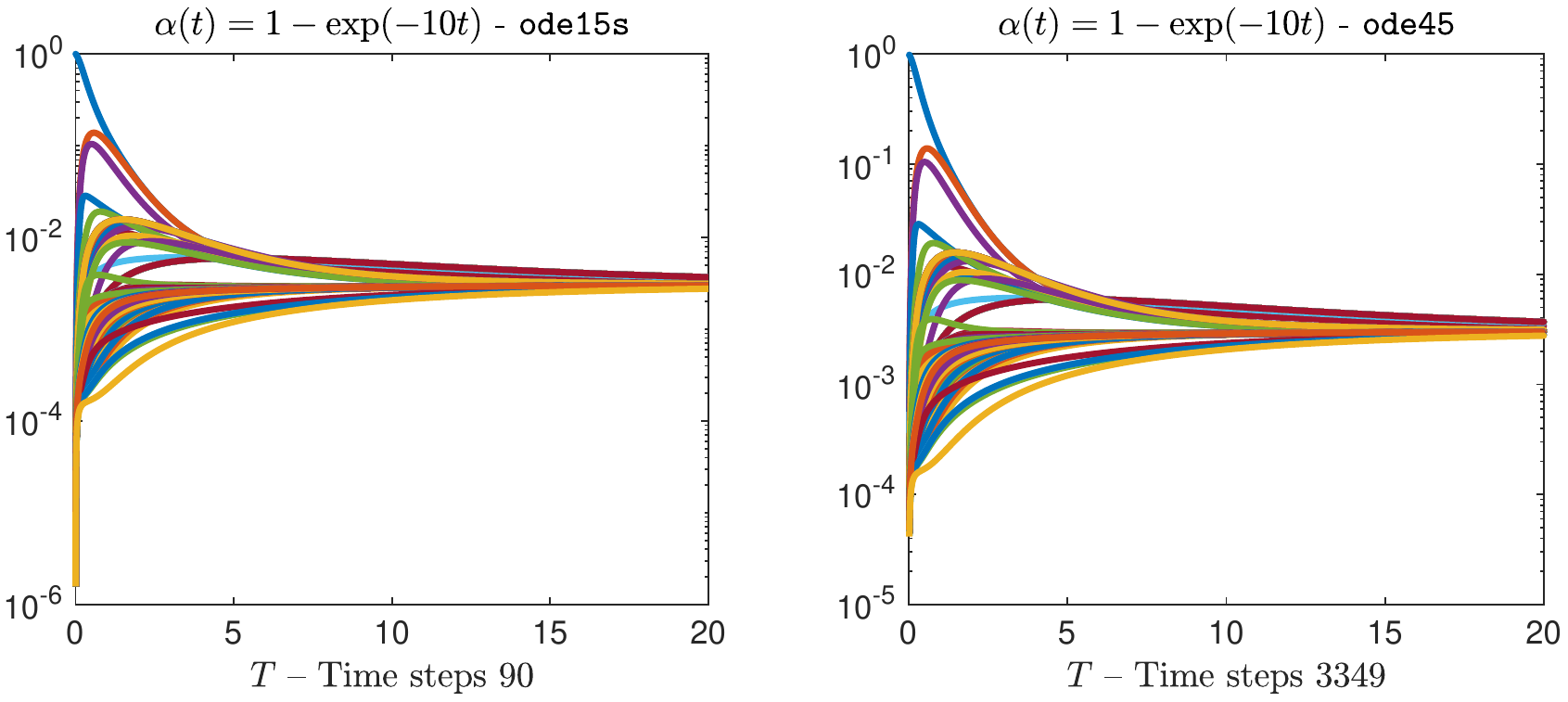}
        \caption{Comparison of short term results of US Airlines in 1997 integrated using
            \texttt{ode45} and \texttt{ode15s}
            $\alpha(t)=1-\exp(-10t)$. Transition probabilities are depicted in logarithmic
            scale to better highlight the differences between the transitions.
            Source of the graph data:~\cite{nr-aaai15}.}
        \label{fig:4USAir97odestiff}
    \end{figure}

    \begin{figure}[htbp]
        \centering
        \includegraphics[width=\columnwidth]{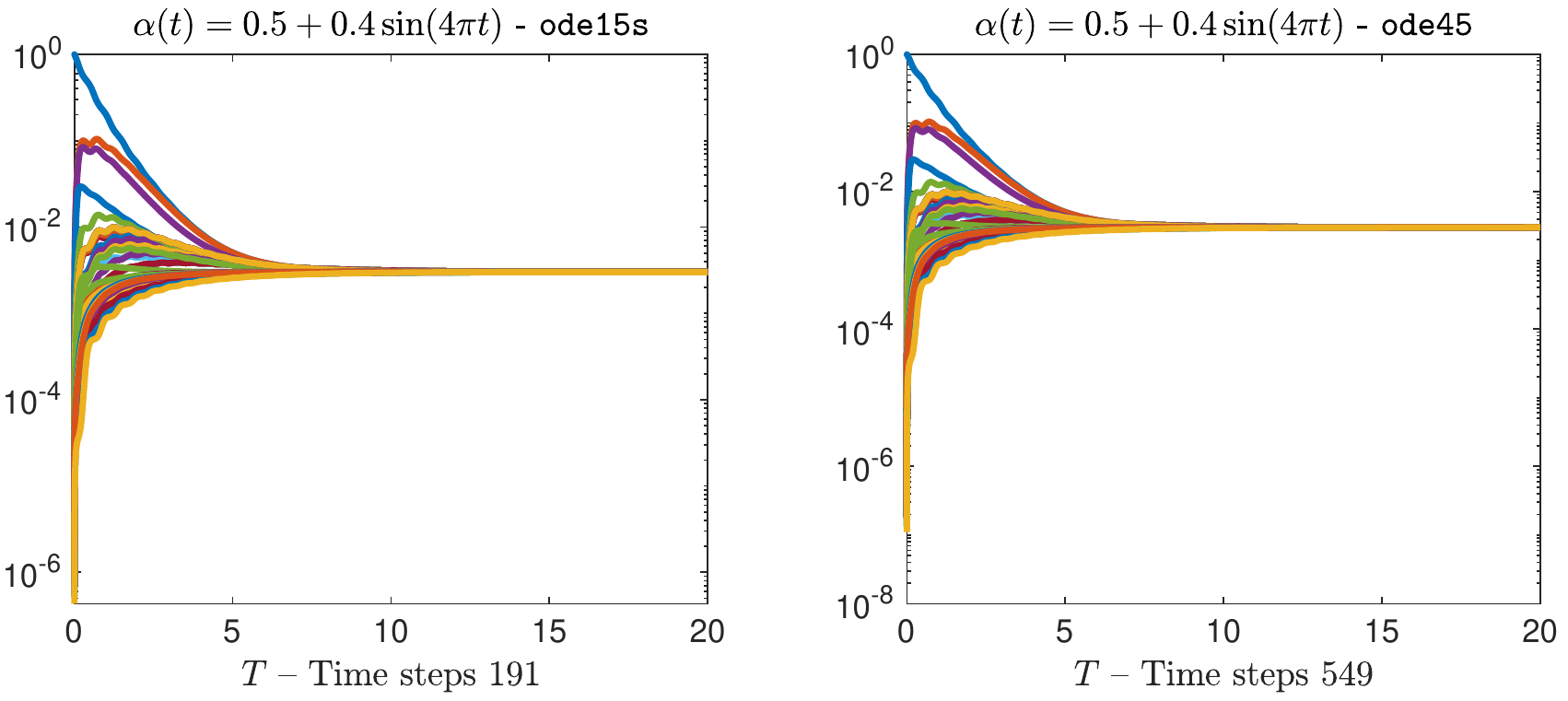}
        \caption{Comparison of short term results of US Airlines in 1997 integrated using \texttt{ode45} and
            \texttt{ode15s},
            $\alpha(t)=0.5+0.4\sin(4\pi t)$. Source of the graph data: \cite{nr-aaai15}.}
        \label{fig:4USAir97odesin-stiff}
    \end{figure}

    To conclude the example, consider using a less regular $\alpha(t)$.
    Specifically, we use a sampling of the function
    $\alpha_{\text{true}}(t) = 0.5 + 0.4 \sin(\pi t/2)$ at points $t_k =
    k$, $k=0,1,\ldots,6$. Then, for integrating the system, we use the
    piecewise polynomial form of the cubic spline interpolating
    $\alpha(t)$ on these data; see Figure~\ref{fig:spline-approx}.
    \begin{figure}[htbp]
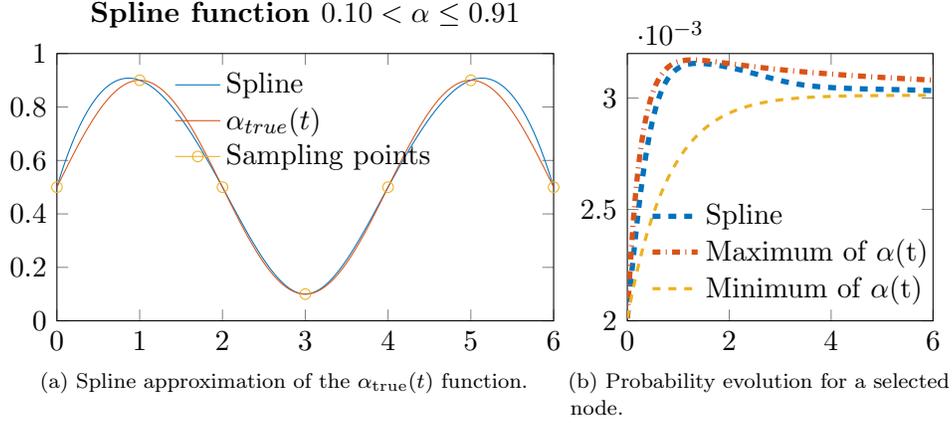

        \centering
        \subfloat[Spline approximation of the $\alpha_{\text{true}}(t)$ function.\label{fig:spline-approx}]{\input{splineapprox.tikz}}
        \subfloat[Probability evolution for a selected node.\label{fig:spline-dynamic}]{\input{spline.tikz}}
        \caption{US Airlines in 1997~\cite{nr-aaai15}. Example of the trajectories obtained by employing
            a function $\alpha(t)$ with reduced regularity. Comparison with the trajectories
            obtained with the constant fractional Laplacian matrix with values the maximum and
            minimum of the $\alpha(t)$ function.}
    \end{figure}
    From Figure~\ref{fig:spline-dynamic}, that represents the
    probability evolution for an arbitrary node of the network, we
    observe again the same behavior shown in the other cases, i.e., the
    probability evolves mimicking the extreme cases, at least in part .

\end{example}

\begin{example}
    Let us now consider a \emph{directed} network. We simulate the
    evolution of the probability distributions on the weighted
    \texttt{cage8} graph from the van Heukelum
    collection~\cite{VANHEUKELUM2002313}. The weights on the adjacency
    matrix describe the transition probabilities between equivalence
    classes of the configurations, for an applied field of $E = (0.1,
    0.1, 0.1)$ and a polymers of $8$ monomers in a cage model of DNA
    electrophoresis. We model the evolution of the probability
    $\mathbf{p}(t)$ for $t \in [0,5]$ for the choices of $\alpha(t)$
    from Figure~\ref{fig:karate1}, and the heat equation dynamics
    from~\eqref{eq:nonautonomous_extension}. That is, we look at the
    transition function $\alpha(t)$ given by the two fixed values, the
    two oscillating functions and the exponential transitions. Since the
    network is now directed, we need to choose which Laplacian matrix we
    adopt. For this case we consider the out-degrees, i.e.,
    $L_{\text{out}}$ in~\eqref{eq:in-and-out-laplacians}.
    \begin{figure}[htbp]
        \centering
        \input{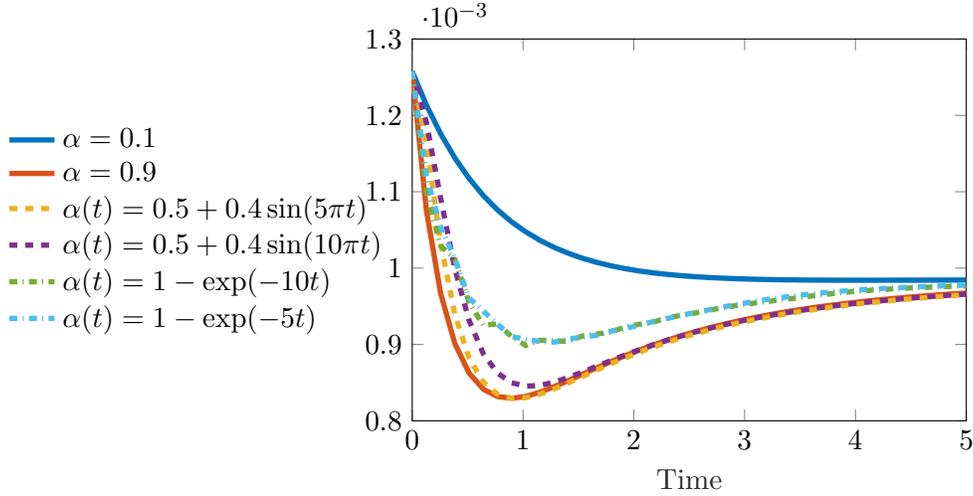}
        \caption{\texttt{Cage8} network~\cite{VANHEUKELUM2002313}. This is a cage model of DNA electrophoresis with $n=1015$ nodes.
            Evolution of the probability $\mathbf{p}_i(t)$ for $i=6$, $t \in [0,5]$ for different choices of $\alpha(t)$ and the Heat equation dynamics from~\eqref{eq:nonautonomous_extension}.\label{fig:cage8}}
    \end{figure}
    To depict a clearer picture, we report in Figure~\ref{fig:cage8} the
    evolution of the probability of a single node. We can observe rapid
    oscillations in the transient phase, while the oscillations settle
    to the steady state more slowly.
\end{example}

\begin{example}
    The last test case we consider is the \emph{EU-Road
        network}~\cite{subelj2011}. This is the international E-road network
    for roads that are mostly located in Europe. The network is
    undirected, with nodes representing cities and links denoting e-road
    between two cities, and is neither \emph{scale-free} nor
    \emph{small-world}. Moreover, since the graph has more than one
    connected components, we restrict the analysis to the largest
    component of the network consisting of 1039 nodes and 1355 edges.

    In the figures in~\ref{fig:euroad.eps}, we report the output and, in
    particular, the evolution of the probability distribution for two
    different runtimes and $\alpha(t)$. We can observe that the
    oscillations reach the steady-state faster than the one in which our
    ``attention span'' rapidly reaches the value 1. Indeed, if we go
    ``one road at a time'' across the whole of Europe, we need
    potentially much more time to explore every place.
    \begin{figure}[htbp]
        \centering
        \includegraphics[width=\columnwidth]{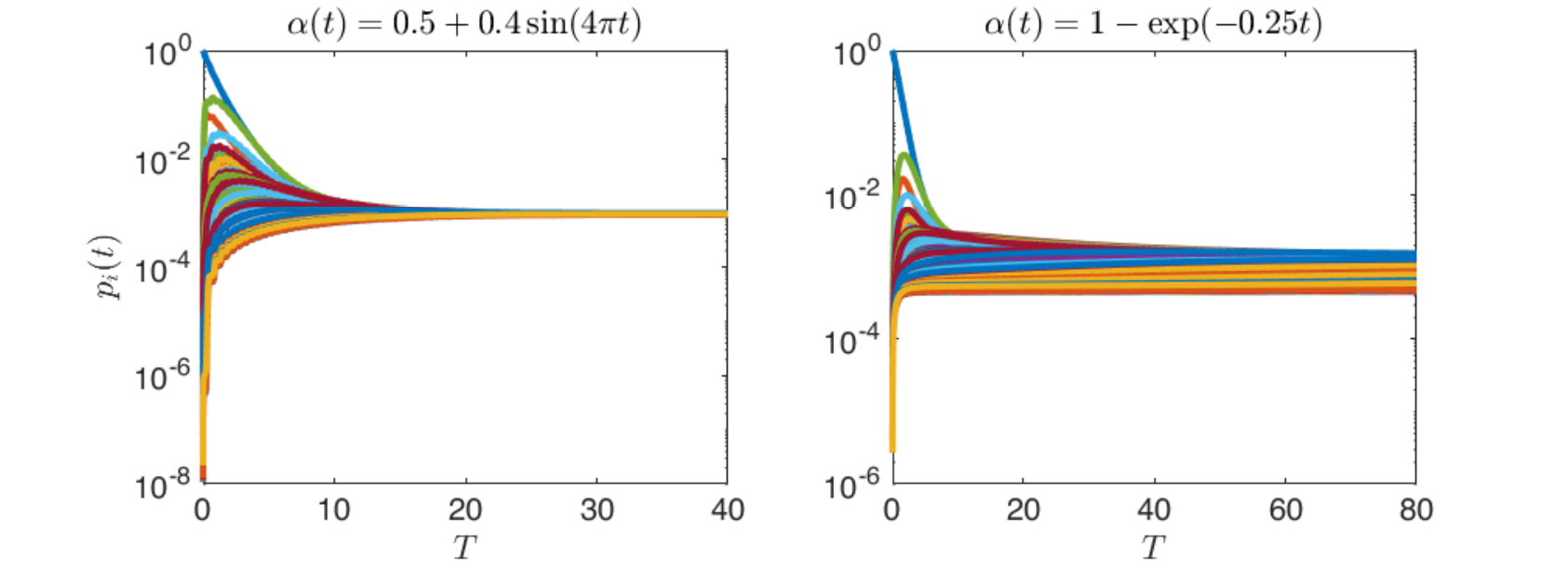}
        \caption{Euroad network. Source of the graph data: \cite{nr-aaai15,subelj2011}.}
        \label{fig:euroad.eps}
    \end{figure}
\end{example}

\section{Conclusions}\label{sec:conclusions}
{We extended the nonlocal exploration of complex
networks utilizing a time-varying function $\alpha=\alpha(t)$ for
the fractional Laplacian matrix used as the Jacobian of an initial
value problem for the evolution of the underlying probability
distributions
(see~\cite{BenziBertacciniDurastanteSimunec,RiascosPhysRevE,PhysRevE.102.022142}).
In particular, we proved existence, uniqueness, and asymptotic
stability properties of the solution of a model based on the heat
equation. Interesting (moderate) stiffness phenomena of the
underlying system of differential equations on some real-world
complex networks are also observed in some numerical experiments
involving simulations of the evolution of the probability
distributions for complex networks.}

There are some directions in which we plan to extend this study. The
first is the generalization to the non-autonomous version of the
transformed $k$-path Laplacian briefly discussed in
Remark~\ref{rmk:k-path-laplacian}, even if the formulation of this
case can be achieved straightforwardly by modifying the $k$-path
operator to
\begin{equation*}
\mathcal{L}_G(\alpha(t);t) =  L_1 + \sum_{k \geq 2} \frac{1}{k^{\alpha(t)}} L_k, \quad \alpha : \mathbb{R}_+ \rightarrow \mathbb{R}_+.
\end{equation*}
This extension does not falls under the observation in
Remark~\ref{rmk:matrix_function_are_simple}, since the matrices
$\mathcal{L}_G(\alpha(t);t)$, even in the symmetric case, cannot be
diagonalized by the same transformation in general.

The second is to (i) extend the approach in~\cite{MR3989621} by
working with a variable order integrator and the singular M-matrix
$L$, by applying some of the techniques proposed
in~\cite{benzi2020rational} and (ii) use computationally efficient
techniques to solve \eqref{eq:thingtosolve} generalizing those
proposed in~\cite{BertacciniDurastanteEnumath19}.

\section*{Acknowledgements}

We would like to thank the referees for their appreciation and
constructive comments. In particular, the
papers~\cite{agarwal,Agarwal2021,Rajchakit2021,Boonsatit2021,Rajchakit2021b}
have been suggested by one of the referees.

\section*{Funding}

The authors are partially supported by INDAM-GNCS and by the INdAM-GNCS project
``Metodi numerici per l'analisi  di modelli innovativi di reti complesse'' CUP E55F22000270001. D. Bertaccini
acknowledges the MIUR Excellence Department Project awarded to the
Department of Mathematics, University of Rome Tor Vergata, CUP
E83C18000100006 and the Tor Vergata University \lq\lq\,Beyond
Borders\rq\rq\ program through the project ASTRID, CUP
E84I19002250005.

\bibliographystyle{tfs}
\bibliography{nonautonomousbibliography}

\end{document}